\documentclass[11pt, letter]{article}

\usepackage[margin=1in]{geometry}
\usepackage{graphicx}
\usepackage{makecell}
\usepackage{booktabs}
\usepackage{multirow}
\usepackage{subcaption}

\usepackage{amsmath, amssymb, amsfonts}
\usepackage{amsthm}

\usepackage{xcolor}
\usepackage{setspace}
\usepackage{lmodern}
\usepackage{titling}
\usepackage{authblk}
\usepackage[authoryear,round]{natbib}

\usepackage[colorlinks=true,
            linkcolor=blue,
            citecolor=blue,
            urlcolor=blue]{hyperref}

\newtheorem{theorem}{Theorem}[section]
\newtheorem{proposition}[theorem]{Proposition}
\newtheorem{lemma}[theorem]{Lemma}

\title{Robust Optimization for Green Ammonia Production}

\author[1]{Karl Zhu}
\author[2]{Yassine Bohafid}
\author[2]{Omar Kadir}
\author[1]{Dimitris Bertsimas}
\affil[1]{Operations Research Center, Massachusetts Institute of Technology, Cambridge, MA, USA}
\affil[2]{OCP Group, Casablanca, Morocco}

\date{\today}
\begin{document}

\maketitle

\begin{abstract}
The central challenge in optimizing green ammonia systems is satisfying the minimum-load requirements of the Haber–Bosch (HB) process under renewable uncertainty. We develop a robust optimization framework consisting of a strategic capacity planning model and an operational flow model under solar and wind uncertainty. The strategic model is a mixed-integer optimization (MIO) problem with flexible HB operating modes, namely hot-idling and shutdowns. To address the resulting computational challenges, we propose a robust scenario-reduction framework that combines $k$-means clustering with robust optimization to generate adversarial renewable trajectories. For the operational model, we develop adaptive robust rolling-horizon formulations under forecast uncertainty. Computational results show that the proposed framework produces feasible capacity plans under out-of-sample simulation, whereas existing approaches based on constraint aggregation fail to satisfy HB minimum-load requirements. Adaptive policies achieve higher ammonia production than static robust policies for a given robustness level, but provide weaker protection against realizations outside the uncertainty set.
\end{abstract}
\noindent\rule{\textwidth}{0.4pt}

\normalsize
\section{Introduction}

Ammonia is the second most produced chemical in the world, with roughly 80\% used in fertilizers \citep{aziz_ammonia_2020}. More recently, it has also attracted attention as a medium for hydrogen transport and energy storage. However, conventional ammonia production is responsible for approximately 2\% of global CO$_2$ emissions \citep{liu_life_2020}, motivating growing interest in \textit{green ammonia} produced using near-100\% renewable electricity.

We are partnering with OCP Group in Morocco—one of the world's largest fertilizer producers—to support a US\$7 billion investment in a green ammonia facility designed to produce one million metric tons of ammonia annually using solar and wind power over a 25-year horizon \citep{atchison_ocp_2023}. Figure~\ref{fig:overview} provides an overview of the system. The central challenge in planning and operating green ammonia facilities is satisfying the operational requirements of the Haber–Bosch (HB) process under renewable intermittency and uncertainty. In particular, HB plants are subject to minimum-load requirements and incur significant costs when frequently shut down and restarted. Capacity planning decisions must therefore remain feasible during prolonged periods of low renewable output, while operational flow decisions must be made under forecast uncertainty.

To model renewable uncertainty, we obtain 25 years of hourly solar and wind capacity factor derived from satellite and weather reanalysis data \citep{european_commission_photovoltaic_2024, setchell_ecmwf_2020}. However, solving a realistic green ammonia planning model directly over a multi-decade hourly dataset is computationally prohibitive when flexible HB operating modes, including hot-idling and shutdowns, are modeled through binary variables.

To address these challenges, we develop a robust optimization framework consisting of a strategic capacity planning model and a rolling-horizon operational model. The strategic model is a mixed-integer optimization (MIO) problem that determines the capacities of solar panels, wind turbines, battery energy storage systems, desalination, electrolyzers, hydrogen storage tanks, and the HB plant to minimize the levelized cost of ammonia (LCOA). To improve computational tractability, we strengthen the formulation through warm-starts and bilinear cuts. We further develop a robust scenario-reduction framework that combines $k$-means clustering with robust optimization to construct worst-case renewable trajectories for capacity planning.

In the operational model, we optimize hourly flows under fixed first-stage capacities in a rolling-horizon setting. We formulate this problem as an adaptive robust optimization model in which flow decisions respond to realized solar and wind uncertainty through affine decision rules. We compare nominal, static robust, and adaptive robust operating policies under forecast uncertainty to quantify the benefits and limitations of adaptive operation.

\begin{figure}[tbp]
\centering
\includegraphics[width=\linewidth]{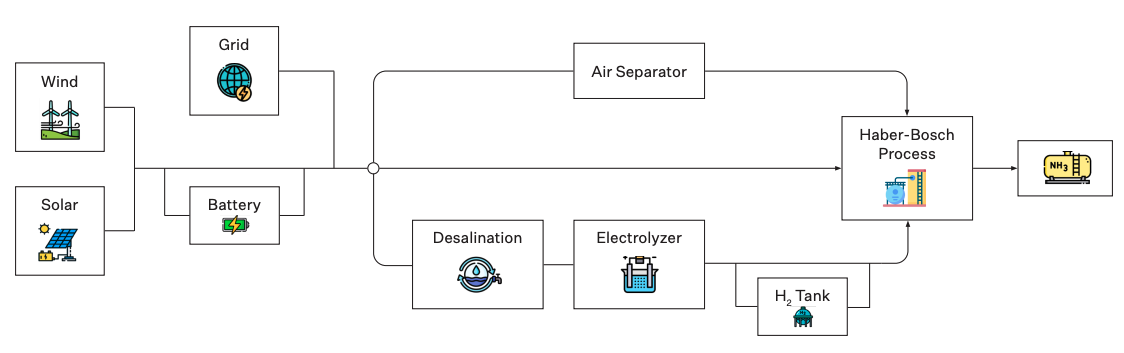}
\caption{Overview of the green ammonia plant. Ammonia (NH$_3$) is produced in a Haber--Bosch (HB) plant using energy, hydrogen (H$_2$), and nitrogen (N$_2$). H$_2$ is produced by splitting desalinated water via electrolysis, and N$_2$ by the air separation unit. All processes are powered by solar panels and wind turbines, with an external grid connection for backup. Electricity and H$_2$ can be stored in batteries and H$_2$ storage tanks, respectively.}
\label{fig:overview}
\end{figure}

\subsection{Related Work}\label{subsec:related-work}

\textbf{Green ammonia optimization.}
There is a growing interest in capacity planning and operation of green power-to-ammonia systems; see \citet{narciso_design_2025} for a recent review. Our framework is most closely related to \citet{salmon_impact_2023}, who develop a capacity planning model minimizing LCOA together with a rolling-horizon operational model. However, their capacity planning model assumes perfect foresight over a single historical year and accounts for renewable uncertainty only indirectly through inflated HB minimum-load requirements. They also do not consider hot-idling and shutdowns. Other recent works incorporate operational flexibility and uncertainty through stochastic or robust optimization. \citet{wang_optimal_2025} develops a two-stage stochastic MIO with HB shutdown decisions, while \citet{yu_optimal_2023} formulate a robust mixed-integer linear-fractional model with box uncertainty sets. However, these approaches rely on a limited number of representative scenarios or simple uncertainty sets and do not address capacity planning using multi-decade, high-resolution renewable datasets. \citet{palys_using_2020} study hydrogen and ammonia-based renewable energy storage using a mixed-integer capacity-planning and scheduling model with temporally clustered representative operating periods.

\textbf{Scenario reduction via clustering.}
Scenario reduction through clustering is widely used in energy-system planning to compress long renewable-generation time series into a tractable set of representative periods \citep{teichgraeber_time-series_2022}. A well-known limitation is that cluster centroids smooth temporal variability and suppress prolonged low-generation events, leading to underestimated capacity and storage requirements \citep{teichgraeber_clustering_2019}. To address this issue, recent work has combined clustering with robust optimization. \citet{bertsimas_decarbonizing_2023} apply robust optimization to $k$-means clustered renewable scenarios in long-term energy capacity expansion, using constraint aggregation to avoid the conservatism of constraint-wise robustness. 
While effective in that setting, constraint aggregation is less suitable for green ammonia systems, where feasibility is highly sensitive to hourly renewable availability due to the HB minimum-load constraints.

\subsection{Contributions and Structure}

Our contributions are threefold.

\begin{enumerate}

\item We formulate a mixed-integer capacity-planning model with flexible HB operation, including hot-idling and shutdown decisions, and strengthen the formulation using warm-starts and bilinear cuts to solve large-scale instances.

\item We develop a robust scenario-reduction framework that combines multi-decade renewable data, $k$-means clustering, and robust optimization to construct representative worst-case renewable trajectories. Unlike centroid-based methods and the constraint-aggregation approach of \citet{bertsimas_decarbonizing_2023}, the resulting scenarios preserve prolonged low-renewable periods that drive HB  minimium-load feasibility and capacity requirements.

\item We develop a rolling-horizon adaptive robust operational model and compare it against nominal and static robust policies. Adaptive policies achieve higher production at comparable robustness levels but provide weaker protection against renewable realizations outside the uncertainty set.

\end{enumerate}

The remainder of the paper is organized as follows. Section~\ref{section:strategic} introduces the strategic capacity planning model. Section~\ref{section:scenario-reduction} presents the robust scenario-reduction framework. Section~\ref{section:operational} develops the operational model under forecast uncertainty. Section~\ref{section:exp-results} reports computational experiments and results. Section~\ref{section:conclusion} concludes.

\section{Strategic Model}\label{section:strategic}

This section presents the strategic capacity planning model. We begin with the standard perfect-foresight formulation, in which uncertain renewable capacity factors are replaced by their historical realizations. We first develop a linear optimization (LO) model, then extend it to a mixed-integer linear optimization (MILO) model by incorporating HB hot-idling and shutdown decisions.

The 25-year historical dataset is split sequentially into 15 years for capacity planning and optimization and 10 years for out-of-sample evaluation. The model is formulated at hourly resolution over the 15-year planning horizon. For simplicity, leap days are omitted, so hours $h \in [24]$, days $d \in [365]$, and years $y \in [15]$.

\subsection{Sets and Parameters}
The sets are defined below. Table~\ref{tab:parameters} describes the parameters, and their values are in Appendix~\ref{app:strategic}.
\begin{description}
    \item[Components] $\mathcal{K} = \{s, w, b, d, e, H_2, HB\}$ denotes the set of infrastructure components subject to capacity planning. These correspond to solar panels ($s$), wind turbines ($w$), batteries ($b$), desalination plant ($d$), electrolyzer ($e$), hydrogen storage tanks ($H_2$), and HB plants ($HB$). The air separation unit is assumed to have sufficient capacity and is therefore excluded from planning.
    \item[Electricity Flows] $\mathcal{I} = \{ 
HB,\; d,\; e,\; a,\; H_2^{\text{comp}}, b^+,\; b^-,\;b^-_e,\; b^-_{HB},\; b^-_d,\; b^-_a,\; b_{H_2^{\text{comp}}},\; b\text{:store}
\}$ denotes the set of electricity flows, including: electricity to the HB plant ($HB$), desalination ($d$), electrolyzer ($e$), air separation ($a$), hydrogen storage compression ($H_2^{\text{comp}}$), battery charging ($b^+$), battery discharging to electrolyzer ($b^-_e$), to HB ($b^-_{HB}$), to desalination ($b^-_d$), to air separation ($b^-_a$), to hydrogen compression ($b^-_{H_2^\text{comp}}$), and battery storage ($b\text{:store}$). We also define $\mathcal{I}_\mathrm{out} = \{a,b^+,d,e,H_2^{\text{comp}},HB\}\subset \mathcal{I}$ as the set of outflows from the renewables. 

    \item[Mass Flows] $\mathcal{J} = \{
HB{:}H_2,\; H_2,\; H_2^+,\; H_2^-,\; H_2{:}\text{store},\; NH_3,\; N_2,\; H_2O
\}$ denotes the set of mass flows, including: total hydrogen input to HB ($HB{:}H_2$), hydrogen produced by the electrolyzer ($H_2$), hydrogen tank inflow/outflow ($H_2^+$, $H_2^-$), hydrogen storage level ($H_2{:}\text{store}$), ammonia produced by HB ($NH_3$), nitrogen produced by the air separator ($N_2$), and water produced by desalination ($H_2O$).
\end{description}

\begin{table}[tbp]
\centering
\caption{Model Parameters and Definitions}
\begin{tabular}{ll}
\toprule
\textbf{Symbol} & \textbf{Description} \\
\midrule
$C_k$ & Unit capital cost of component $k \in \mathcal{K}$ \\
$c_k$ & Marginal operational cost of component $k \in \mathcal{K}$ \\
$\dot{v}_s^{h,d,y},\ \dot{v}_w^{h,d,y}$ & Hourly solar and wind capacity factors at time $(h,d,y)$ \\
$\eta_\text{RTE}$ & Battery round-trip efficiency \\
$\eta_\text{heat}$ & Battery heat-loss efficiency \\
$\Lambda_b,\ \Lambda_{H_2}$ & Storage flow limits for battery and hydrogen tanks \\
$\delta_k$ & Yearly degradation factor for $k \in \{b, e\} \subset \mathcal{K}$ \\
$\phi^e_j$ & Tons of material produced by one MWh, for material $j \in \{H_2^\text{comp}, H_2, N_2, H_2O, NH_3\}$ \\
$\phi^m_{j,j'}$ & Mass ratio of $j$ per unit mass of $j'$; for $(j,j') \in \{(H_2, H_2O), (NH_3, N_2), (NH_3, H_2)\}$ \\
\bottomrule
\end{tabular}
\label{tab:parameters}
\end{table}

\subsection{Decision Variables}
For capacity decisions, we have $\boldsymbol{K} =(K_k)_{k \in \mathcal{K}}$, measured in MW (solar, wind, electrolyzers), MWh (battery storage), m$^3$/year (desalination), tons (hydrogen storage), and tons of NH$_3$/year (HB plant). For hourly flow decisions, we have electricity $\boldsymbol{x} =(x_i^{h,d,y})_{i\in\mathcal{I}, (h,d,y)}$ in MWh, and mass $\boldsymbol{w} =(w_j^{h,d,y})_{j\in\mathcal{J}, (h,d,y)}$ in tons. 
\subsection{Objective}\label{section:objective}

We minimize the levelized cost of ammonia (LCOA). All capacity investments are assumed to occur at year $y=0$. Let $\boldsymbol{C}=(C_k)_{k \in \mathcal{K}}$, and $(C_{\text{trans}},C_{\text{port}},C_{\text{dev}},C_{\text{cont}})$ denote the capital cost coefficients, fixed transportation and port costs, and development and contingency factors, respectively. The capital expenditure is
\(
\text{CAPEX}
=
\big(
\boldsymbol{C}^\top\boldsymbol{K}
+
C_{\text{trans}}
+
C_{\text{port}}
\big)
(1+C_{\text{dev}})
(1+C_{\text{cont}}).
\)
Similarly, let $\boldsymbol{c}=(c_k)_{k \in \mathcal{K}}$ denote the annual operating cost coefficients, with fixed costs $(c_{\text{trans}},c_{\text{port}})$ and factors $(c_{\text{dev}},c_{\text{cont}})$. The operating expenditure in year $y$ is \(\text{OPEX}_y
=
\big(
\boldsymbol{c}^\top\boldsymbol{K}
+
c_{\text{trans}}
+
c_{\text{port}}
\big)
(1+c_{\text{dev}})
(1+c_{\text{cont}}). \) Given a discount rate $\gamma$ and annual ammonia production $\kappa_y$, the LCOA is
\begin{equation}
\label{eq:lcoa}
\text{LCOA}
=
\frac{
\text{CAPEX}
+
\sum_y \frac{\text{OPEX}_y}{(1+\gamma)^y}
}{
\sum_y \frac{\kappa_y}{(1+\gamma)^y}
}.
\end{equation}
Under OCP's production contracts, annual production is fixed at $\kappa_y=1$ MT/year for all $y$. The resulting LCOA \eqref{eq:lcoa} is linear function with respect to the decision variables.

\subsection{Perfect Foresight LO Formulation}\label{section:lp-formulation}
\begin{subequations}
\small
\begin{align}
\min_{\boldsymbol{K,x,w}} &\quad \text{LCOA},&~\eqref{eq:lcoa} \nonumber \\
\text{s.t.} \quad & \ \dot{v}_s^{h,d,y} K_s + \dot{v}_w^{h,d,y} K_w \ge \sum_{i \in \mathcal{I}_\text{out}} x_{i}^{h,d,y}, & \forall(h,d,y), \label{eq:lp-model-energy-balance}  \\
& x_{b:\text{store}}^{1,1,1} = S_{b,\text{init}} K_b, & \label{eq:lp-model-init-battery} \\
& w_{H_2:\text{store}}^{1,1,1} = S_{H_2,\text{init}} K_{H_2}, & \label{eq:lp-model-init-h2} \\
& x_{b:\text{store}}^{\tau} \le x_{b:\text{store}}^{\tau^-} + \sqrt{\eta_\text{RTE}} x_{b^+}^{\tau}  - \frac{1}{\sqrt{\eta_\text{RTE}}} x_{b^-}^{\tau}, & \forall \tau\in\mathcal{T}\setminus\{(1,1,1)\} \neq (1,1,1), & \label{eq:lp-model-battery-update} \\
& x_{b^-_e}^{h,d,y} + x_{b^-_{HB}}^{h,d,y} + x_{b^-_d}^{h,d,y} + x_{b^-_a}^{h,d,y} + x_{b^-_{H_2^\text{comp}}}^{h,d,y} \le \eta_\text{heat}  x_{b^-}^{h,d,y}, & \forall (h,d,y), \label{eq:lp-model-battery-efficiency} \\
& w_{H_2:\text{store}}^{\tau} \le w_{H_2:\text{store}}^{\tau^-} + w_{H_2^+}^{\tau} - w_{H_2^-}^{\tau}, & \forall \tau\in\mathcal{T}\setminus\{(1,1,1)\}, \label{eq:lp-model-h2-update} \\
& x_{b:\text{store}}^{h,d,y} \le \delta_b^{y-1} S_{b,\text{max}} K_b, & \forall (h,d,y), \label{eq:lp-model-battery-cap} \\
& w_{H_2:\text{store}}^{h,d,y} \le S_{H_2,\text{max}} K_{H_2}, & \forall (h,d,y), \label{eq:lp-model-h2-cap} \\
& x_{b^-}^{h,d,y} \le \delta^{y-1}_b \Lambda_b K_b, & \forall (h,d,y), \label{eq:lp-model-battery-discharge-limit} \\
& w_{H_2^+}^{h,d,y}, w_{H_2^-}^{h,d,y} \le \Lambda_{H_2}, & \forall (h,d,y), \label{eq:lp-model-h2-flowlimit} \\
& w_{HB:H_2}^{h,d,y} \le w_{H_2}^{h,d,y} + w_{H_2^-}^{h,d,y} - w_{H_2^+}^{h,d,y}, & \forall (h,d,y), \label{eq:lp-model-h2-balance-hb} \\
& w_{H_2^-}^{h,d,y} \le \phi_{H_2}^e (x_{H_2^\text{comp}}^{h,d,y} + x_{b^-_{H_2^\text{comp}}}^{h,d,y}), & \forall (h,d,y), \label{eq:lp-model-h2-compression} \\
& w_{H_2}^{h,d,y} \le \phi_{H_2}^e (x_e^{h,d,y} + x_{b^-_e}^{h,d,y}), & \forall (h,d,y), \label{eq:lp-model-h2-from-electricity} \\
& w_{H_2}^{h,d,y} \le \phi_{H_2,H_2O}^m w_{H_2O}^{h,d,y}, & \forall (h,d,y), \label{eq:lp-model-h2-from-water} \\
& w_{N_2}^{h,d,y} \le \phi_{N_2}^e (x_a^{h,d,y} + x_{b^-_a}^{h,d,y}), & \forall (h,d,y), \label{eq:lp-model-n2} \\
& w_{H_2O}^{h,d,y} \le \phi_{H_2O}^e (x_d^{h,d,y} + x_{b^-_d}^{h,d,y}), & \forall (h,d,y), \label{eq:lp-model-h2o} \\
& w_{NH_3}^{h,d,y} \le \phi^e_{NH_3} (x_{HB}^{h,d,y} + x_{b^-_{HB}}^{h,d,y}), & \forall (h,d,y), \label{eq:lp-model-nh3-from-electricity} \\
& w_{NH_3}^{h,d,y} \le \phi_{NH_3,N_2}^m w_{N_2}^{h,d,y}, & \forall (h,d,y), \label{eq:lp-model-nh3-from-n2} \\
& w_{NH_3}^{h,d,y} \le \phi_{NH_3,H_2}^m w_{HB:H_2}^{h,d,y}, & \forall (h,d,y), \label{eq:lp-model-nh3-from-h2} \\
& x_e^{h,d,y} + x_{b^-_e}^{h,d,y} \le \delta_e^{y-1} K_e, & \forall (h,d,y), \label{eq:lp-model-electrolyzer-cap} \\
& w_{H_2O}^{h,d,y} \le \frac{K_d}{8760}, & \forall (h,d,y), \label{eq:lp-model-h2o-cap} \\
& w_{NH_3}^{h,d,y} \le \frac{K_{HB}}{8760}, & \forall (h,d,y), \label{eq:lp-model-nh3-cap} \\
& w_{NH_3}^{h,d,y} \ge \frac{\lambda_{HB}^\text{min} K_{HB}}{8760}, & \forall (h,d,y), \label{eq:lp-model-min-load-nh3} \\
& x_{HB}^{h,d,y} + x_{b^-_{HB}}^{h,d,y} \ge \frac{\lambda_{HB}^{\min} K_{HB}}{8760  \phi^e_{NH_3}} + \lambda^e_{HB}K_{HB}, & \forall (h,d,y), \label{eq:lp-model-min-energy-hb} \\
& \sum_{d=1}^{365} \sum_{h=1}^{24} w_{NH_3}^{h,d,y} \ge \kappa_y, & \forall y, \label{eq:lp-model-annual-target} \\
& \boldsymbol{K,x,w} \ge \boldsymbol{0}. \label{eq:lp-model-nonneg}
\end{align}

\end{subequations}
The model minimizes LCOA \eqref{eq:lcoa} subject to
\eqref{eq:lp-model-energy-balance}--\eqref{eq:lp-model-nonneg}. Hourly energy
balance is enforced by \eqref{eq:lp-model-energy-balance}. Initial battery and
hydrogen storage levels are specified in
\eqref{eq:lp-model-init-battery}--\eqref{eq:lp-model-init-h2}. Storage dynamics
are governed by \eqref{eq:lp-model-battery-update}--\eqref{eq:lp-model-h2-update},
where $\mathcal{T}=[24]\times[365]\times[15]$ denotes the ordered set of
hourly time indices and $\tau^{-}$ denotes the immediately preceding hour of
$\tau\in\mathcal{T}\setminus\{(1,1,1)\}$, with rollover across day and year
boundaries. Storage capacity and flow-rate limits are imposed by
\eqref{eq:lp-model-battery-cap}--\eqref{eq:lp-model-h2-flowlimit}.
Constraint~\eqref{eq:lp-model-h2-balance-hb} defines hydrogen availability for
the HB process, and \eqref{eq:lp-model-h2-compression} accounts for
the energy required to compress hydrogen into storage. The conversion of
electricity and feedstocks into hydrogen, nitrogen, water, and ammonia is
modeled by \eqref{eq:lp-model-h2-from-electricity}--\eqref{eq:lp-model-nh3-from-h2},
while capacity limits are imposed by
\eqref{eq:lp-model-electrolyzer-cap}--\eqref{eq:lp-model-nh3-cap}. The minimum
HB load and energy requirements are enforced through
\eqref{eq:lp-model-min-load-nh3}--\eqref{eq:lp-model-min-energy-hb}. Finally,
\eqref{eq:lp-model-annual-target} enforces the annual ammonia production
target, and \eqref{eq:lp-model-nonneg} imposes nonnegativity on all capacity
and flow variables.\subsection{Hot-Idling and Shutdowns: MIO Extensions}\label{section:mip}

The minimum-load requirements
\eqref{eq:lp-model-min-load-nh3}--\eqref{eq:lp-model-min-energy-hb}
can substantially increase renewable generation and storage requirements. To improve operational flexibility, modern HB reactors may support \emph{hot-idling} and \emph{shutdowns}. In hot-idle mode, the reactor remains at operating temperature but produces no ammonia. Under normal operation, the plant must satisfy the minimum-load requirement $\lambda_{HB}^{\min}$ and the fixed energy consumption $\lambda_{HB}^{e}K_{HB}$; under hot-idling, ammonia production is suspended while the reactor continues to consume $\lambda_{HB}^{\mathrm{idle}}$ of full-load energy in addition to the fixed energy requirement. Shutdowns allow the reactor to be taken completely offline. We assume at most $N_{\mathrm{shutdowns}}$ shutdowns per year and an $L$-hour transition delay period for both shutdown and restart operations. These operating modes introduce binary decisions, extending the model to a MILO.

For the shutdown and restart logic, we re-index hours within each year by $t=24(d-1)+h$ and write $x_i^{t,y}$ and $w_j^{t,y}$ for the flow variables at the hour corresponding to $(h,d,y)$. For each $(t,y)$, define the binary variables
\[
\begin{aligned}
o^{t,y} &= 1 && \text{if the HB plant is hot-idle, and 0 otherwise},\\
p^{t,y} &= 1 && \text{if the HB plant is shut down, and 0 otherwise},\\
q^{t,y} &= 1 && \text{if the HB plant is hot-idle or shut down, and 0 otherwise},\\
r^{t,y} &= 1 && \text{if a restart is initiated, and 0 otherwise},\\
s^{t,y} &= 1 && \text{if a shutdown is initiated, and 0 otherwise}.
\end{aligned}
\]
Each year begins online, so we set $p^{0,y}=0$. We also use the convention that $s^{\tau,y}=r^{\tau,y}=0$ for all $\tau\le 0$.

\begin{subequations}\label{eq:hot_idle_shutdown_constraints}
\begin{align}
o^{t,y} + p^{t,y} &\le 1 && \forall (t,y) \label{eq:z_eta_sum} \\
q^{t,y} &\ge o^{t,y} && \forall (t,y) \label{eq:delta_ge_z} \\
q^{t,y} &\ge p^{t,y} && \forall (t,y) \label{eq:delta_ge_eta} \\
q^{t,y} &\le o^{t,y} + p^{t,y} && \forall (t,y) \label{eq:delta_le_sum} \\
q^{t,y} = 0 &\implies \frac{\lambda_{HB}^{\min} K_{HB}}{8760} \le w_{NH_3}^{t,y} \le \frac{K_{HB}}{8760} && \forall (t,y) \label{eq:delta_zero_prod_bounds} \\
q^{t,y} = 0 &\implies x_{HB}^{t,y} + x_{b^-_{HB}}^{t,y} \ge \frac{\lambda_{HB}^{\min}K_{HB}}{8760 \cdot \phi^e_{NH_3}} + \lambda_{HB}^e K_{HB} && \forall (t,y) \label{eq:delta_zero_power_bound} \\
o^{t,y} = 1 &\implies w_{NH_3}^{t,y} = 0 && \forall (t,y) \label{eq:hotidle_zero_prod} \\
o^{t,y} = 1 &\implies x_{HB}^{t,y} + x_{b^-_{HB}}^{t,y} \ge \frac{\lambda_{HB}^{\mathrm{idle}}K_{HB}}{8760 \cdot \phi^e_{NH_3}} + \lambda^e_{HB} K_{HB}
&& \forall (t,y) \label{eq:hotidle_energy_requirement} \\
p^{t,y} = 1 &\implies 
\begin{cases}
w_{NH_3}^{t,y} = 0,\\
x_{HB}^{t,y} = 0,\\
x_{b^-_{HB}}^{t,y} = 0
\end{cases} && \forall (t,y) \label{eq:shutdown_zero_all} \\
s^{t,y}+r^{t,y} &\le 1 && \forall (t,y) \label{eq:no_simultaneous_transition} \\
s^{t,y} &\le 1-p^{t,y} && \forall (t,y) \label{eq:shutdown_only_if_online} \\
r^{t,y} &\le p^{t,y} && \forall (t,y) \label{eq:restart_only_if_shutdown} \\
\sum_{l=0}^{L-1}\left(s^{t-l,y}+r^{t-l,y}\right) &\le 1 && \forall (t,y) \label{eq:no_overlapping_transitions} \\
p^{t,y} &= p^{t-1,y}+s^{t-L,y}-r^{t-L,y} && \forall (t,y) \label{eq:shutdown_state_transition} \\
s^{t,y}=r^{t,y} &=0 && \forall t>8760-L,\ y \label{eq:no_late_transitions} \\
\sum_{t=1}^{8760} s^{t,y} &\le N_\text{shutdowns} && \forall y \label{eq:max_shutdowns} \\
\boldsymbol{o,p,q,r,s} &\in \{0,1\}^{8760\times 15}. \label{eq:binary}
\end{align}
\end{subequations}

Constraints~\eqref{eq:z_eta_sum}--\eqref{eq:delta_le_sum} define the mutually exclusive operating modes: normal operation, hot-idling, and shutdown, encoded through the auxiliary variable $q^{t,y}$. The minimum-load and energy requirements during normal operation are enforced by \eqref{eq:delta_zero_prod_bounds}--\eqref{eq:delta_zero_power_bound}. Hot-idling is modeled by \eqref{eq:hotidle_zero_prod}--\eqref{eq:hotidle_energy_requirement}, in which ammonia production is suspended while energy is consumed to maintain operating temperature. Constraint~\eqref{eq:shutdown_zero_all} represents complete shutdown. Constraints~\eqref{eq:no_simultaneous_transition}--\eqref{eq:restart_only_if_shutdown} ensure that shutdowns are initiated only when the plant is online and restarts only when the plant is shut down. Constraint~\eqref{eq:no_overlapping_transitions} prevents overlapping shutdown and restart transitions. Constraint~\eqref{eq:shutdown_state_transition} updates the shutdown state after the $L$-hour transition period. Constraint~\eqref{eq:no_late_transitions} avoids boundary effects at the end of each year, and \eqref{eq:max_shutdowns} limits the number of shutdowns per year.

The original minimum-load constraints \eqref{eq:lp-model-min-load-nh3} and \eqref{eq:lp-model-min-energy-hb} are replaced by \eqref{eq:delta_zero_prod_bounds}--\eqref{eq:hotidle_energy_requirement}.
\subsection{Formulation Strengthening}

Introducing binary operating modes substantially increases the complexity of the strategic model. We employ two formulation-strengthening techniques to improve tractability.

\paragraph{Warm start.}
Setting all binary variables $\boldsymbol{o,p,q,r,s}$ to zero recovers the LO formulation in Section~\ref{section:lp-formulation}. Since the MILO extends the LO model by introducing optional hot-idling and shutdown operating modes, every feasible LO solution remains feasible for the MILO when the binary variables are fixed at zero. We therefore solve the LO model first and use the resulting solution to warm-start the MILO, immediately providing a valid upper bound.

\paragraph{Bilinear cuts.}
We introduce three bilinear cuts corresponding to constraints \eqref{eq:lp-model-nh3-cap}, \eqref{eq:lp-model-min-load-nh3}, and \eqref{eq:lp-model-min-energy-hb}. These cuts are logically equivalent to the associated indicator constraints and substantially strengthen the relaxation. For each $(t,y)$, we add
\begin{subequations}
\begin{align}
w_{NH_3}^{t,y}
&\le
(1-q^{t,y})
\frac{K_{HB}}{8760},
\tag{2v'}
\\
w_{NH_3}^{t,y}
&\ge
(1-q^{t,y})
\frac{\lambda_{HB}^{\min}K_{HB}}{8760},
\tag{2w'}
\\
x_{HB}^{t,y}
+
x_{b^-_{HB}}^{t,y}
&\ge
\frac{
\Bigl[
\lambda_{HB}^{\min}(1-p^{t,y})
+
\bigl(\lambda_{HB}^{\mathrm{idle}}-\lambda_{HB}^{\min}\bigr)
o^{t,y}
\Bigr]
K_{HB}
}{
8760\,\phi_{NH_3}^{e}
}
+
(1-p^{t,y})
\lambda_{HB}^{e}K_{HB}.
\tag{2x'}
\end{align}
\end{subequations}

Although these constraints transform the strategic model into a mixed-integer quadratically constrained optimization problem, they substantially strengthen the LO relaxation. Computational results in Section~\ref{exp:mip} demonstrate the resulting improvements in solution quality and tractability.
 \subsection{Limitations of the Perfect-Foresight Model}

The perfect-foresight formulation replaces uncertain renewable generation with historical solar and wind capacity factors. While this approach is standard in the green ammonia literature, it introduces look-ahead bias by assuming full knowledge of future renewable realizations. To partially mitigate this issue, we use 15 years of hourly renewable data for capacity planning rather than a single representative year, thereby exposing the model to a broader range of operating conditions.

However, incorporating flexible HB operation through hot-idling and shutdown decisions makes the resulting MILO computationally intractable at this scale. Consequently, some form of scenario reduction is required to obtain a tractable planning model.

\section{Scenario Reduction via Clustering and Robust Optimization}
\label{section:scenario-reduction}

We model renewable uncertainty directly through the solar and wind capacity factors rather than indirectly by heuristically inflating the HB minimum-load requirement $\lambda_{HB}^{\min}$. However, solving the strategic model over 15 years of hourly renewable data becomes computationally intractable once hot-idling and shutdown decisions are included. We therefore develop a scenario-reduction framework that compresses the renewable data while preserving the adverse generation patterns that drive capacity requirements.

The framework has two stages. First, we cluster daily renewable trajectories using $k$-means and select representative years whose cluster distribution matches the full training dataset. Second, because cluster centroids smooth prolonged low-generation events, we solve a robust optimization problem within each cluster to construct adversarial representative trajectories. These trajectories are then used in the strategic model to obtain tractable and conservative capacity decisions.

\subsection{Clustering and Representative Year Selection}

Each day is encoded as a 48-dimensional vector by concatenating its 24-hour solar and wind capacity factor profiles. We perform $k$-means clustering with $k = 20$ on the 15-year training dataset. We set $k=20$ as a trade-off between temporal coverage and computational tractability: increasing $k$ preserves more renewable variability but enlarges the resulting strategic MIO. For each cluster $c$ and hour $h$, the centroid $\overline{v}^{h,c}_{\texttt{source}}$ is defined as the mean capacity factor across all days assigned to that cluster. The set of days assigned to cluster $c \in [20]$ is
\[
\mathcal{D}_c := \left\{(d,y) \in [365] \times [15] : (d,y) \text{ is assigned to cluster } c \right\}.
\]

To further reduce the size of the planning problem, we select a subset of $n$ representative years whose average cluster distribution closely matches that of the full 15-year dataset. Let $N$ denote the total number of training years, $P \in [0,1]^k$ the empirical cluster distribution over all $N$ years, and $C \in [0,1]^{k \times N}$ the matrix whose column $y$ gives the empirical cluster distribution in year $y$. Let $z_y$ indicate whether year $y$ is selected, let $\Gamma \in \mathbb{R}_+^{k \times k}$ denote the transport matrix, and let $D \in \mathbb{R}_+^{k \times k}$ denote the Euclidean distance matrix between cluster centroids. We solve the MILO model:
\begin{subequations}\label{eq:wasserstein}
\begin{align}
\min_{\boldsymbol \Gamma, \boldsymbol{z}}  \quad & \sum_{i=1}^k \sum_{j=1}^k D_{ij} \, \Gamma_{ij}, \label{eq:wasserstein-objective} \\
\text{s.t.} \quad 
& \sum_{j=1}^k \Gamma_{ij} = P_i, \quad \forall i \in [k], \label{eq:wasserstein-source} \\
& \sum_{i=1}^k \Gamma_{ij} = \frac{1}{n} \sum_{y=1}^N C_{j,y} z_y, \quad \forall j \in [k], \label{eq:wasserstein-target} \\
& \sum_{y=1}^N z_y = n, \label{eq:wasserstein-cardinality} \\
& \Gamma_{ij} \ge 0, \quad \forall (i,j) \in [k]^2, \label{eq:wasserstein-nonnegativity} \\
& z_y \in \{0,1\}, \quad \forall y \in [N]. \label{eq:wasserstein-binary}
\end{align}
\end{subequations}
The objective \eqref{eq:wasserstein-objective} minimizes the 1-Wasserstein distance between the empirical cluster distribution of the selected years and that of the full dataset. Constraints \eqref{eq:wasserstein-source}--\eqref{eq:wasserstein-target} define the transport plan, \eqref{eq:wasserstein-cardinality} selects exactly $n$ years, and \eqref{eq:wasserstein-nonnegativity}--\eqref{eq:wasserstein-binary} impose flow nonnegativity and integrality.

After selecting the representative years, we denote the chosen subset by $\mathcal{Y}$. Since these years are not necessarily consecutive, storage dynamics are treated independently across selected years: the initial storage constraints \eqref{eq:lp-model-init-battery}--\eqref{eq:lp-model-init-h2} are enforced for each $y \in \mathcal{Y}$, while the storage transition constraints \eqref{eq:lp-model-battery-update} and \eqref{eq:lp-model-h2-update} are not linked across year boundaries. Degradation factors are applied according to the original chronological index of each selected year.
\subsection{Adversarial Scenario Generation with Robust Optimization}\label{section:strategic-robust}
Cluster centroids smooth prolonged periods of low renewable generation. Consequently, replacing all renewable trajectories in cluster $c$ with the centroid profile $\overline{v}^{h,c}_{\texttt{source}}$---which we refer to as the \emph{nominal} model---can underestimate the renewable, storage, and process capacities required to maintain production. This is particularly problematic for green ammonia systems, where extended low-generation events (\emph{dunkelflaute}) interact with the HB minimum-load requirements.

To preserve these adverse operating conditions, we construct a worst-case renewable trajectory within each cluster using robust optimization. The resulting finite-scenario approach retains the computational advantages of clustering while recovering the low-generation patterns that drive capacity planning decisions.
\subsubsection{Uncertainty Sets}

Given the mean capacity factor \( \overline{v}_\texttt{source}^{h,c} \) for hour \( h \), cluster \( c \), and renewable \( \texttt{source} \in \{s,w\} \), define the perturbation and the perturbation vector as
\[
u_\texttt{source}^{h,d,y} := v_\texttt{source}^{h,d,y} - \overline{v}_\texttt{source}^{h,c},\, \forall  (d,y) \in \mathcal{D}_c; \quad \boldsymbol{u}^c_\texttt{source} := (u^{h,d,y}_\texttt{source})_{h,\,(d,y)\in \mathcal{D}_c},
\]
respectively. We adopt the Central Limit Theorem (CLT)-inspired polyhedral uncertainty set proposed by \cite{bertsimas_decarbonizing_2023}:
\begin{subequations} \label{eq:clt-uncertainty}
\begin{align}
\mathcal{U}^{c}_{\texttt{source}} = \Big\{ \boldsymbol{u}^c_\texttt{source} \,:\;
    & \quad 0 \le \overline{v}^{h,c}_{\texttt{source}} + u^{h,d,y}_{\texttt{source}} \le 1, & \forall h ,\,(d,y) \in \mathcal{D}_c,
    \label{eq:clt-capacity-bound} \\
    & \quad \left|u^{h,d,y}_{\texttt{source}}\right| \le \rho \sigma^{h,c}_{\texttt{source}},  &\forall h ,\,(d,y) \in \mathcal{D}_c,
    \label{eq:clt-hourly-box} \\
    & \quad \left| \sum_{(d,y) \in \mathcal{D}_c} u^{h,d,y}_{\texttt{source}} \right| \le \Gamma \sigma^{h,c}_{\texttt{source}} \sqrt{|\mathcal{D}_c|}, & \forall h,
        & \label{eq:clt-clt} \\ & \quad \left| u^{h,d,y}_{\texttt{source}} - u^{h-1,d,y}_{\texttt{source}} \right| \le \Delta \sigma^{h,c}_{\Delta,\texttt{source}}, &\forall h \ge 2,\,(d,y) \in \mathcal{D}_c\label{eq:clt-smooth}  \Big\}
\end{align}
\end{subequations}

\noindent
where $\sigma^{h,c}_{\texttt{source}}$ denotes the empirical standard deviation of the capacity factor and $\sigma^{h,c}_{\Delta,\texttt{source}}$ denotes the empirical standard deviation of the first differences. The parameters $\rho$, $\Gamma$, and $\Delta$ control the hourly deviation \eqref{eq:clt-hourly-box}, aggregate deviation \eqref{eq:clt-clt}, and temporal smoothness \eqref{eq:clt-smooth}, respectively.

A one-sided variant is obtained by additionally imposing
$\boldsymbol u^c_{\texttt{source}}\le \boldsymbol 0$.
Under this restriction, \eqref{eq:clt-clt} reduces to an $\ell_1$-norm budget constraint, which together with \eqref{eq:clt-hourly-box} yields the classical budget uncertainty set of \cite{bertsimas_price_2004}. We therefore interpret the budget uncertainty set as a one-sided CLT uncertainty set that only permits adverse deviations from the cluster centroid.
\subsubsection{Existing Approaches: Constraint-Wise and Aggregated Robustness}\label{section:constraintwise-aggregated}
The uncertain renewable capacity factors enter the strategic model through the hourly energy-balance constraint \eqref{eq:lp-model-energy-balance}. A direct robust counterpart requires this constraint to hold for all admissible perturbations of the solar and wind capacity factors:

\begin{equation}\label{eq:full_robust}
     (\overline{v}_s^{h,c} + u_s^{h,d,y}) K_s +  (\overline{v}_w^{h,c} + u_w^{h,d,y}) K_w \ge \sum_{i \in \mathcal{I}_\text{out}} x_{i}^{h,d,y}, \quad \forall (u_s^{h,d,y}, u_w^{h,d,y}) \in \mathcal{U}^{h,c}_s \times \mathcal{U}^{h,c}_w
\end{equation}
where $c$ denotes the cluster assigned to day $(d,y)$. The uncertainty set $\mathcal{U}^c_{\texttt{source}}$ is coupled across hours through the aggregate deviation constraint \eqref{eq:clt-clt} and the temporal smoothness constraint \eqref{eq:clt-smooth}. Directly enforcing \eqref{eq:full_robust} is overly conservative because standard robust reformulations effectively break this coupling and reduce the model to independent constraint-wise uncertainty sets \citep{marandi_when_2018}. For a general discussion of coupled uncertainty sets, see \citet{bertsimas_benefit_2024}.

An alternative is \emph{constraint aggregation} \citep{bertsimas_decarbonizing_2023}, which retains the nominal hourly constraints
\begin{equation}\label{eq:nominal_hourly}
      \overline{v}_s^{h,c} K_s + \overline{v}_w^{h,c} K_w \ge \sum_{i \in\mathcal{I}_\text{out}} x_i^{h,d,y}, \quad \forall h,d,y
\end{equation}
with the following aggregated robust constraints for each cluster $c$, 
\begin{align*}
    &\sum_{h=1}^{24}\sum_{(d,y)\in \mathcal{D}_c}\left[ (\overline{v}_s^{h,c} + u_s^{h,d,y}) K_s +  (\overline{v}_w^{h,c} + u_w^{h,d,y}) K_w\right] \\ = & |\mathcal{D}_c| \sum_{h=1}^{24}\left( \overline{v}_s^{h,c}  K_s +  \overline{v}_w^{h,c} K_w \right)+  \sum_{h=1}^{24}\sum_{(d,y)\in \mathcal{D}_c}\left( u_s^{h,d,y} K_s +   u_w^{h,d,y} K_w\right) \\\ge&  \sum_{h=1}^{24}\sum_{(d,y)\in \mathcal{D}_c}\sum_{i \in \mathcal{I}_\text{out}} x_{i}^{h,d,y}, \quad \forall (\boldsymbol{u}^{c}_s,\boldsymbol{u}^{c}_w) \in \mathcal{U}^{c}_s \times \mathcal{U}^{c}_w.
\end{align*}  
This is equivalent to satisfying the worst-case scenario:
\begin{align}\label{eq:aggregation_wc}
    & |\mathcal{D}_c|\sum_{h=1}^{24} \left( \overline{v}_s^{h,c}  K_s +  \overline{v}_w^{h,c} K_w \right) + \min_{\boldsymbol{u}_s^c \in \mathcal{U}_s^{c}} \sum_{h=1}^{24}\sum_{(d, y)\in \mathcal{D}_c} u_s^{h,d,y} K_s + \min_{\boldsymbol{u}_w^c \in \mathcal{U}_w^{c}} \sum_{h=1}^{24}\sum_{(d, y)\in \mathcal{D}_c} u_w^{h,d,y} K_w \nonumber \\ \ge&\sum_{h=1}^{24}\sum_{(d,y)\in \mathcal{D}_c} \sum_{i \in \mathcal{I}_\text{out}} x_{i}^{h,d,y} , \quad \forall c.
\end{align}

Constraint aggregation reduces the conservatism of the full robust formulation, but it may be insufficiently protective for green ammonia planning. The aggregated constraint protects total renewable generation over an entire cluster rather than enforcing robustness at the hourly level. Therefore, a local period of low generation can violate hourly energy balance even when the aggregate constraint remains satisfied. Furthermore, the robust counterpart provides only an aggregate protection guarantee and does not directly produce an interpretable renewable trajectory that can be used for scenario generation.

Despite this limitation, the aggregated formulation possesses a useful structural property. Since all capacity decisions are made before operations and satisfy $K_s,K_w \ge 0$, the inner adversarial problem separates from the capacity variables:
\begin{equation}\label{eq:rearrange}
    \min_{{\boldsymbol{u}}^{c}_\texttt{source} \in \mathcal{U}_\texttt{source}^{c}}
    \sum_{h=1}^{24}\sum_{(d,y)\in\mathcal{D}_c}
    u_\texttt{source}^{h,d,y} K_\texttt{source}
    =
    K_\texttt{source}
    \min_{\boldsymbol{u}^{c}_\texttt{source}\in\mathcal{U}_\texttt{source}^{c}}
    \sum_{h=1}^{24}\sum_{(d,y)\in\mathcal{D}_c}
    u_\texttt{source}^{h,d,y},
    \quad \forall c,\texttt{source}.
\end{equation}

As a result, the worst-case perturbation can be computed independently of the optimal capacity decisions. Proposition~\ref{prop:rc-aggregation} characterizes the robust counterpart of the aggregated formulation and yields a conservative closed-form approximation. More importantly, it identifies a worst-case perturbation vector that can subsequently be used to construct adversarial renewable trajectories.

\begin{proposition}\label{prop:rc-aggregation}
For each cluster $c$ and \texttt{source}, let
\begin{equation}\label{eq:closed-form}
    \boldsymbol{\hat{u}}^{c}_\texttt{source}
    =
    \arg\min_{\boldsymbol{u}^{c}_\texttt{source}\in\mathcal{U}^{c}_\texttt{source}}
    \sum_{h=1}^{24}\sum_{(d,y)\in\mathcal{D}_c}
    u_\texttt{source}^{h,d,y}.
\end{equation}

For any $K_s,K_w \ge 0$, the robust counterpart of
\eqref{eq:aggregation_wc} is
\begin{equation}\label{eq:aggregation_rc}
   |\mathcal{D}_c|
   \sum_{h=1}^{24}
   \left(
   \overline{v}_s^{h,c}K_s
   +
   \overline{v}_w^{h,c}K_w
   \right)
   +
   \sum_{h=1}^{24}\sum_{(d,y)\in\mathcal{D}_c}
   \left(
   \hat{u}_s^{h,d,y}K_s
   +
   \hat{u}_w^{h,d,y}K_w
   \right)
   \ge
   \sum_{h=1}^{24}\sum_{(d,y)\in\mathcal{D}_c}
   \sum_{i\in\mathcal{I}_{\mathrm{out}}}
   x_i^{h,d,y},
   \quad \forall c.
\end{equation}

Furthermore, a conservative approximation is given by
\begin{equation}\label{eq:conservative-aggregation}
    |\mathcal{D}_c|
    \sum_{h=1}^{24}
    \left(
    \overline{v}_s^{h,c}K_s
    +
    \overline{v}_w^{h,c}K_w
    \right)
    -
    \Gamma\sqrt{|\mathcal{D}_c|}
    \sum_{h=1}^{24}
    \left(
    \sigma_s^{h,c}K_s
    +
    \sigma_w^{h,c}K_w
    \right)
    \ge
    \sum_{h=1}^{24}\sum_{(d,y)\in\mathcal{D}_c}
    \sum_{i\in\mathcal{I}_{\mathrm{out}}}
    x_i^{h,d,y},
    \quad \forall c,
\end{equation}
that is, satisfying \eqref{eq:conservative-aggregation} implies satisfying \eqref{eq:aggregation_rc}. If
\[
\left|
\sum_{(d,y)\in\mathcal{D}_c}
\hat{u}^{h,d,y}_{\texttt{source}}
\right|
=
\Gamma
\sigma^{h,c}_{\texttt{source}}
\sqrt{|\mathcal{D}_c|},
\quad
\forall h,c,\texttt{source},
\]
then \eqref{eq:conservative-aggregation} is exact.
\end{proposition}

\begin{proof}
Since $K_s,K_w \ge 0$, the separation property
\eqref{eq:rearrange} holds. Substituting the minimizers
$\hat{\boldsymbol{u}}_s^c$ and $\hat{\boldsymbol{u}}_w^c$
into \eqref{eq:aggregation_wc} yields the robust counterpart
\eqref{eq:aggregation_rc}. The conservative approximation
\eqref{eq:conservative-aggregation} follows by replacing the optimal
perturbation values with the CLT bound \eqref{eq:clt-clt}. When the
minimizers attain this bound at every hour, the approximation is exact.
\end{proof}
\subsubsection{Adversarial Scenario Generation}

The worst-case perturbation vector identified in Proposition~\ref{prop:rc-aggregation} naturally defines an adversarial renewable trajectories. We use this observation to construct a finite-scenario robustification scheme that lies between full constraint-wise robustness and constraint aggregation. Unlike constraint aggregation, the proposed approach preserves hourly energy-balance protection. Unlike the full robust counterpart, it replaces the uncertainty set with a single representative worst-case realization.

Specifically, let $\hat{\boldsymbol{u}}^c_\texttt{source}$ be the solution of \eqref{eq:closed-form}. We enforce this worst-case realization directly in the hourly energy-balance constraints \eqref{eq:lp-model-energy-balance}:
\begin{equation}\label{eq:one_wc}
     (\overline{v}_s^{h,c} + \hat{u}_s^{h,d,y}) K_s
     +
     (\overline{v}_w^{h,c} + \hat{u}_w^{h,d,y}) K_w
     \ge
     \sum_{i \in \mathcal{I}_\text{out}} x_{i}^{h,d,y},
     \quad
     \forall h \in [24],\,
     (d,y) \in \mathcal{D}_c .
\end{equation}

This formulation admits a natural interpretation as generating a synthetic worst-case renewable trajectory for each cluster. The resulting trajectories can then be used as representative scenarios in the strategic planning model, yielding a compressed yet adversarial representation of the original renewable dataset.

Proposition~\ref{thm:hierarchy} establishes the robustness hierarchy of the proposed approach under the budget uncertainty set. The finite-scenario formulation is less conservative than the full robust counterpart while remaining more conservative than constraint aggregation. This hierarchy relies on the one-sided structure of the budget uncertainty set and does not generally hold for the CLT.
\begin{proposition}[Robustness hierarchy]\label{thm:hierarchy}
Let $\mathcal{F}_{\mathrm{full}}$, $\mathcal{F}_{\mathrm{one}}$, and $\mathcal{F}_{\mathrm{agg}}$ denote the feasible regions of the full robust counterpart, the finite-scenario formulation \eqref{eq:one_wc}, and the constraint aggregation formulation, respectively. Under the budget uncertainty set,
\[
\mathcal{F}_{\mathrm{full}}
\subseteq
\mathcal{F}_{\mathrm{one}}
\subseteq
\mathcal{F}_{\mathrm{agg}}.
\]
\end{proposition}

\begin{proof}
The full robust constraint (\ref{eq:full_robust}) imposes feasibility for all realizations. Therefore, satisfying (\ref{eq:full_robust}) implies satisfying one worst-case realization (\ref{eq:one_wc}), so $\mathcal{F}_\text{full} \subseteq \mathcal{F}_\text{one}$. In constraint aggregation, we impose the nominal hourly constraints (\ref{eq:nominal_hourly}) and worst-case on the aggregated constraints (\ref{eq:aggregation_wc}). 
For the budget uncertainty set, the worst-case satisfies $\hat{\boldsymbol{u}}_\texttt{source}^c \le \boldsymbol{0}, \ \forall c,\texttt{source}$, and is a minimizer. Therefore, satisfying (\ref{eq:one_wc}) implies satisfying (\ref{eq:nominal_hourly}) and (\ref{eq:aggregation_wc}), hence $\mathcal{F}_\text{one} \subseteq \mathcal{F}_\text{agg}$.
\end{proof}

\subsubsection{Diversifying Adversarial Trajectories}
Problem~\eqref{eq:closed-form} generally admits multiple optimal solutions. However, solving it directly with a standard solver often produces nearly identical trajectories, likely because the solver’s initialization tends to yield similar patterns across hours. This reduces scenario diversity and weakens coverage. To address this, we propose a two-stage randomized procedure. First, we solve~\eqref{eq:closed-form} to obtain the optimal objective value $\zeta^*$ and enforce this value as a constraint, restricting the search to the set of optimal solutions. Next, to generate diverse trajectories, we draw independent samples $\boldsymbol{\mu}^c_\texttt{source} := ({\mu}^{h,d,y}_\texttt{source})_{h,\,(d,y)\in \mathcal{D}_c}$ uniformly from the bounds~\eqref{eq:clt-capacity-bound}--\eqref{eq:clt-hourly-box}. Specifically, for each cluster $c$, hour $h$, day $(d,y) \in \mathcal{D}_c$, and renewable $\texttt{source}$, we sample from $\texttt{Uniform}([L^{h,d,y}_\texttt{source},\, U_\texttt{source}^{h,d,y}])$, where \[
L^{h,d,y}_\texttt{source} \;=\; \max\{-\rho\sigma^{h,c}_\texttt{source},\; -\overline v^{h,c}_\texttt{source}\}, 
\qquad 
U^{h,d,y}_\texttt{source} \;=\;
\begin{cases}
\min\{1-\overline{v}^{h,c}_\texttt{source},\;\rho\sigma^{h,c}_\texttt{source}\}, & \text{if CLT}, \\[6pt]

0, & \text{if budget}.
\end{cases}
\]
Next, each sample $\boldsymbol{\mu}^c_\texttt{source}$ is projected onto the feasible set by solving
\[
\min_{\boldsymbol{u}^c_\texttt{source} \in \mathcal{U}^c_\texttt{source}} \quad \|\boldsymbol{u}^c_\texttt{source} - \boldsymbol{\mu}^c_\texttt{source}\|_2^2, \quad \text{s.t.} \sum_{h,(d,y) \in \mathcal{D}_c}  u_\texttt{source}^{h,d,y} = \zeta^*.\]
The quadratic projection ensures a unique randomized adversarial trajectory is generated. Figure~\ref{fig:cluster1_worst} illustrates this randomized finite scenario approach for one of the clusters with the CLT uncertainty set with the robust parameters set to \( (\rho, \Gamma, \Delta) = (3, 3, 3) \).
\begin{figure}[tbp]
    \centering
    \includegraphics[width=\linewidth]{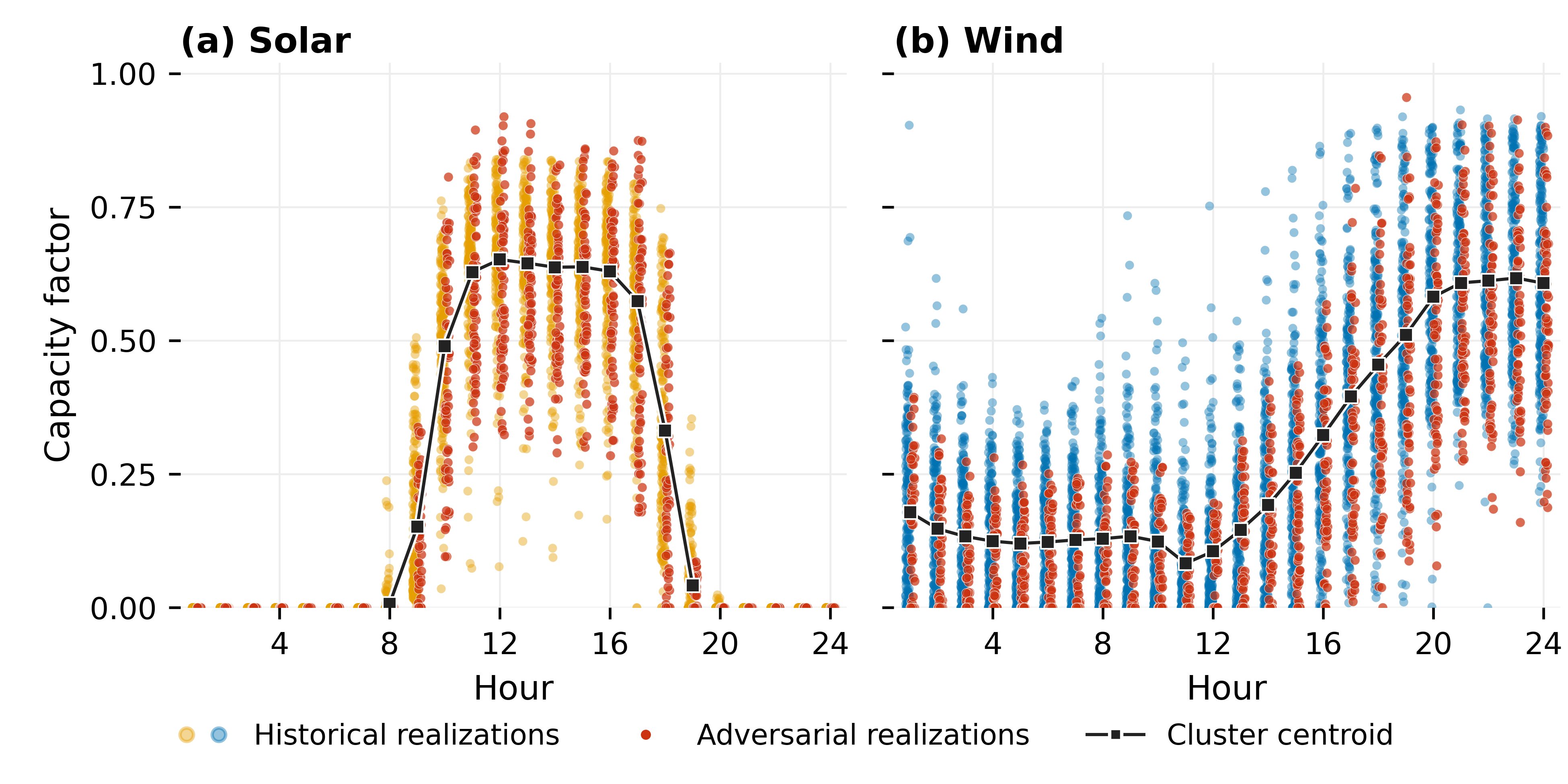}
\caption{Historical renewable generation realizations (orange and blue points), cluster centroids (black line), and diversified adversarial realizations (red points) for a representative cluster derived from 15 years of hourly solar and wind capacity factor data. The adversarial realizations are generated from the CLT uncertainty set with parameters $(\rho,\Gamma,\Delta)=(3,3,3)$.}
\label{fig:cluster1_worst} 
\end{figure}
\section{Operational Model}\label{section:operational}

This section develops the operational model used to dispatch flows under fixed capacities $\boldsymbol{K}$ obtained from the strategic model. Unlike the strategic model, which uses historical renewable realizations for capacity planning, the operational model uses short-term solar and wind forecasts in a rolling-horizon setting. We first describe the modifications needed to convert the strategic model into an operational dispatch model, and then develop nominal, static robust, and adaptive robust formulations under forecast uncertainty.
\subsection{Modifications to the Strategic Model}

The operational model retains the overall structure and constraints of the strategic model (Section~\ref{section:lp-formulation}), but differs in three respects. First, it is implemented in a rolling-horizon framework using renewable forecasts rather than historical realizations. Second, limited grid imports are permitted to preserve feasibility under forecast errors. Third, the objective is modified to balance ammonia production, storage preservation, and grid usage.

\subsubsection{Rolling Horizon}

We assume access to short-term forecasts of solar and wind generation. This motivates a \emph{rolling-horizon} framework in which operational decisions are optimized over a prediction horizon of length $T$, but only the initial portion of the solution is implemented before the model is re-solved using updated forecasts. Let $t \in [T]$ index time within the prediction horizon. Time is therefore re-indexed from $(h,d,y)$, which denotes absolute time in the strategic model, to $t$, which denotes time relative to the current optimization horizon.

The forecast solar and wind capacity factors are denoted by $\overline v_s^t$ and $\overline v_w^t$, respectively. At the end of each control horizon, the realized battery and hydrogen storage levels are carried forward as the initial conditions for the next optimization, replacing \eqref{eq:lp-model-init-battery} and \eqref{eq:lp-model-init-h2} to ensure temporal consistency.

\subsubsection{Grid Intake}

The strategic model disallows grid imports to ensure that installed renewable and storage capacities are sufficient to satisfy production requirements. During operations, however, forecast errors and unexpected renewable shortfalls may render the system infeasible. To preserve feasibility, we introduce a grid-import variable $x_{\mathrm{grid}}^t \ge 0$ and modify the energy-balance constraint \eqref{eq:lp-model-energy-balance} to

\begin{equation}\label{eq:operational-balance}
   \sum_{i \in \mathcal{I}\text{out}} x_i^t - x_{\mathrm{grid}}^t
   \le
   v_s^t K_s + v_w^t K_w,
   \qquad \forall t,
\end{equation}
where $v_{\texttt{source}}^t$ denotes the realized capacity factor of renewable source $\texttt{source}$. To discourage reliance on external electricity, grid imports are penalized heavily in the objective.

\subsubsection{Objective}

With capacities fixed, minimizing LCOA is equivalent to maximizing ammonia production. However, maximizing production alone over a finite prediction horizon can incentivize excessive depletion of battery and hydrogen inventories, leaving the system vulnerable to forecast errors in subsequent horizons. To mitigate this effect, we reward storage preservation through positive coefficients $k_B$ and $k_H$ for battery and hydrogen storage, respectively. Together with a large penalty $M_{\mathrm{grid}}$ on grid imports, the objective becomes

\begin{equation}\label{eq:operational-objective}
    \max_{\boldsymbol{x},\boldsymbol{w}}
    \quad
    \sum_{t=1}^{T}
    \Big(
    w_{NH_3}^{t}
    + k_B x_{b:\mathrm{store}}^{t}
    + k_H w_{H_2:\mathrm{store}}^{t}
    - M_{\mathrm{grid}} x_{\mathrm{grid}}^{t}
    \Big).
\end{equation}
Unlike standard rolling-horizon formulations that reward only terminal storage levels, we reward storage throughout the horizon. This provides an additional buffer against renewable uncertainty and helps maintain feasibility of the HB minimum-load constraints.

Finally, since ammonia production is already maximized through the objective, the annual production constraint \eqref{eq:lp-model-annual-target} is removed.

\subsection{Adaptive Robust Optimization}

We develop an adaptive robust optimization model in which operational flow decisions adjust to renewable forecast errors. The operational model allows flow decisions to adapt to realized deviations in solar and wind generation. We adopt affine decision rules to obtain a tractable approximation of the fully adaptive problem. We first introduce the affine decision rules, then define the uncertainty set, derive the affinely adaptive robust counterparts (AARCs), and finally establish probabilistic guarantees for constraint satisfaction.

\subsubsection{Affine Decision Rules}

Fully adaptive policies are generally intractable. A common tractable approximation is to restrict the adaptive decisions to affine decision rules \citep{ben-tal_adjustable_2004}:
\[
\boldsymbol{x}(\boldsymbol{u})
=
\boldsymbol{z}_x
+
\boldsymbol{V}_x\boldsymbol{u},
\qquad
\boldsymbol{w}(\boldsymbol{u})
=
\boldsymbol{z}_w
+
\boldsymbol{V}_w\boldsymbol{u},
\]
where
$\boldsymbol{z}_x \in \mathbb{R}^{|\mathcal{I}|T}$,
$\boldsymbol{z}_w \in \mathbb{R}^{|\mathcal{J}|T}$,
$\boldsymbol{V}_x \in \mathbb{R}^{|\mathcal{I}|T \times 2T}$,
and
$\boldsymbol{V}_w \in \mathbb{R}^{|\mathcal{J}|T \times 2T}$
are decision variables.

To reduce complexity, we further restrict attention to \emph{stage-wise} affine policies, where each flow at time $t$ depends only on the uncertainty realized at time $t$. This reduces the number of affine coefficients from fully affine policy's $O(T^2)$ to $O(T)$. Equivalently, each row of $\boldsymbol{V}_x$ and $\boldsymbol{V}_w$ contains nonzero entries only in the columns corresponding to the same time step. Thus, for each $i\in\mathcal{I}$, $j\in\mathcal{J}$, and $t\in[T]$,
\begin{subequations}
\begin{align}
x_i^t(u_s^t,u_w^t)
&=
z_i^t
+
V_x^{(i,t),(s,t)}u_s^t
+
V_x^{(i,t),(w,t)}u_w^t,
\label{eq:adr-x}
\\
w_j^t(u_s^t,u_w^t)
&=
z_j^t
+
V_w^{(j,t),(s,t)}u_s^t
+
V_w^{(j,t),(w,t)}u_w^t.
\label{eq:adr-w}
\end{align}
\end{subequations}
The grid-intake variable $x_{\mathrm{grid}}^t$ is not assigned an explicit affine decision rule. Instead, it is determined implicitly as the minimum grid import required to satisfy \eqref{eq:operational-balance} after the renewable capacity factors are realized. Equivalently, $x_{\mathrm{grid}}^t$ represents the renewable shortfall needed to maintain feasibility under the realized uncertainty. 

\subsubsection{Uncertainty Set}

Let
\(
\overline{\boldsymbol{v}}
=
(\overline{v}_s^1,\ldots,\overline{v}_s^T,\,
\overline{v}_w^1,\ldots,\overline{v}_w^T)^\top
\)
denote the vector of point forecasts. For each renewable source
$\texttt{source}\in\{s,w\}$ and time step $t\in[T]$, let
$u_{\texttt{source}}^t$ denote the forecast error, so that the realized
capacity factor is
\(
v_{\texttt{source}}^t
=
\overline{v}_{\texttt{source}}^t
+
u_{\texttt{source}}^t.
\) Since capacity factors are bounded between $0$ and $1$, the perturbations satisfy
\(
0\le
\overline{v}_{\texttt{source}}^t
+
u_{\texttt{source}}^t
\le
1,
\)
which induces a natural box constraint on the forecast errors.

If the uncertainty set is compact and constraint-wise separable, static robust and adaptive robust formulations attain the same optimal value \citep{marandi_when_2018}. Consequently, some form of coupling across time is required for adaptivity to provide value. Among the most common coupled uncertainty sets are budget and ellipsoidal uncertainty sets. Budget uncertainty sets preserve linearity and are therefore attractive in large-scale mixed-integer settings, whereas ellipsoidal uncertainty sets yield second-order cone programs and admit stronger probabilistic guarantees under both approximately Gaussian and distribution-free forecast errors \citep{bertsimas_probabilistic_2021}. If hot-idling and shutdown decisions were incorporated into the operational model, budget uncertainty sets would likely be preferable because they preserve linearity, while ellipsoidal uncertainty sets would lead to mixed-integer second-order cone programs with substantially greater computational burden. Since we focus on the continuous operational model, we adopt an ellipsoidal-box uncertainty set
\begin{subequations}\label{eq:operational-uset}
\begin{align}
\mathcal{U}
=
\Big\{
\boldsymbol{u}
=
(u_s^1,\ldots,u_s^T,\,
u_w^1,\ldots,u_w^T)^\top
:\;
&\;
0
\le
\overline{v}_{\texttt{source}}^t
+
u_{\texttt{source}}^t
\le
1,
&&
\forall t,\texttt{source},
\label{eq:uset-capacity-bound}
\\
&\;
\left\|
\Sigma^{-1/2}\boldsymbol{u}
\right\|_2
\le
\Omega
\Big\},
\label{eq:uset-ellipsoid}
\end{align}
\end{subequations}
where $\Sigma \succeq 0$ is the covariance matrix of the forecast errors and $\Omega \ge 0$ is a robustness parameter. Probabilistic guarantees for calibrating $\Omega$ are provided in Appendix~\ref{app:prob-guarantee}.
\subsubsection{Robust Counterpart}

Substituting the affine decision rules \eqref{eq:adr-x}--\eqref{eq:adr-w}
into the objective and constraints yields a family of robust constraints
that depend explicitly on the uncertainty vector $\boldsymbol{u}$.
We first reformulate the max--min objective through an epigraph variable
$\psi$, obtaining
\[
\max_{\boldsymbol{z}_x,\boldsymbol{z}_w,
      \boldsymbol{V}_x,\boldsymbol{V}_w,\psi}
\quad \psi
\]
subject to
\begin{equation}\label{eq:operational-epigraph}
\min_{\boldsymbol{u}\in\mathcal U}
f(\boldsymbol{z}_x,\boldsymbol{z}_w,
  \boldsymbol{V}_x,\boldsymbol{V}_w,\boldsymbol{u})
\ge \psi,
\end{equation}
where $f(\cdot)$ denotes the operational objective
\eqref{eq:operational-objective} after substitution of the affine decision
rules.

To derive the affinely adaptive robust counterpart (AARC), we write each
constraint in the generic form
\begin{equation}\label{eq:constraint-general}
\overline{\boldsymbol{a}}_l^\top \boldsymbol{K}
+
\delta_l^\psi \psi
+
\delta_l^e\,\boldsymbol{u}^\top
\boldsymbol{K}_{\mathrm{source}}^t
+
\boldsymbol{d}_{x,l}^\top
(\boldsymbol{z}_x+\boldsymbol{V}_x\boldsymbol{u})
+
\boldsymbol{d}_{w,l}^\top
(\boldsymbol{z}_w+\boldsymbol{V}_w\boldsymbol{u})
\le b_l,
\quad
\forall \boldsymbol{u}\in\mathcal U,
\end{equation}
where $\delta_l^\psi=1$ if constraint $l$ corresponds to the epigraph
constraint \eqref{eq:operational-epigraph} and $0$ otherwise,
while $\delta_l^e=1$ if $l$ corresponds to the energy-balance constraint
\eqref{eq:operational-balance} and $0$ otherwise. Proposition~\ref{prop:aarc} states the AARC.
\begin{proposition}[AARC derivation]\label{prop:aarc}
Let
\(\boldsymbol L=-\overline{\boldsymbol v}, \ \boldsymbol U=\boldsymbol 1-\overline{\boldsymbol v}.\)
Under the ellipsoidal-box uncertainty set
\eqref{eq:operational-uset}, the affinely adaptive robust counterpart of
\eqref{eq:constraint-general} is
\begin{align}
&
\overline{\boldsymbol a}_l^\top \boldsymbol K
+
\delta_l^\psi \psi
+
\boldsymbol d_{x,l}^\top \boldsymbol z_x
+
\boldsymbol d_{w,l}^\top \boldsymbol z_w
+
\boldsymbol U^\top\boldsymbol\beta
-
\boldsymbol L^\top\boldsymbol\alpha
\nonumber\\
&\quad
+
\Omega
\left\|
\Sigma^{1/2}
\left(
\delta_l^e
\boldsymbol K_{\mathrm{source}}^t
+
\boldsymbol V_x^\top\boldsymbol d_{x,l}
+
\boldsymbol V_w^\top\boldsymbol d_{w,l}
-
\boldsymbol\beta
+
\boldsymbol\alpha
\right)
\right\|_2
\le b_l,
\label{eq:aarc}
\end{align}
with auxiliary variables
$\boldsymbol\alpha,\boldsymbol\beta\ge\boldsymbol 0$.
\end{proposition}
\begin{proof}
The result follows from conic duality and the support function of the ellipsoidal-box uncertainty set. The full derivation is provided in Appendix~\ref{app:proof-aarc}.
\end{proof}

\section{Computational Experiments}\label{section:exp-results}

We evaluate the proposed framework through three computational studies. First, we investigate the value of incorporating HB operational flexibility through hot-idling and shutdown decisions. We assess both the impact of these features on optimal capacity planning and the effectiveness of warm-starting and bilinear cuts in improving MILO performance. Second, we compare the capacity decisions obtained from the proposed scenario-reduced robust formulations against those from a perfect-foresight benchmark. The resulting designs are then evaluated under out-of-sample operating conditions. Third, we study operational performance under forecast uncertainty. Using fixed capacities from the strategic model, we compare nominal, static robust, and adaptive robust formulations with respect to ammonia production, external grid dependence, and robustness to forecast errors.

\subsection{Impact of Hot-Idling and Shutdowns on Capacity Planning}\label{exp:mip}

Introducing hot-idling and shutdown decisions into the HB process transforms the strategic model into a MILO. This experiment investigates whether the additional complexity improves capacity-planning decisions. To isolate this effect, we consider a one-year perfect-foresight instance and minimize CAPEX only. We also evaluate the impact of warm-starting and the bilinear cuts developed in Section~\ref{section:mip}.
\subsubsection{Solver Performance with Warm Starts and Bilinear Cuts}

We consider three formulations: \textit{hot-idle} (shutdowns disabled, $\boldsymbol p=\boldsymbol r=\boldsymbol s=\boldsymbol 0$), \textit{shutdown} (hot-idling disabled, $\boldsymbol o=\boldsymbol 0$), and \textit{full} (both hot-idling and shutdowns enabled). For each formulation, we test progressively stronger families of bilinear cuts derived in Section~\ref{section:mip}: no cuts (baseline MILO), 1 cut (adds (2w')), 2 cuts (adds (2v')), and 3 cuts (adds (2x')). The maximum number of shutdowns per year is set to $N_{\mathrm{shutdown}}=10$, and the transition delay is $L = 24$ hours.

All models are solved using Gurobi 11.0.0 on the MIT SuperCloud, with 32 CPU threads, 192 GB RAM, and an Intel Xeon Platinum 8260 processor. The solver settings are a relative optimality gap of 0.01\% and a time limit of 72 hours.
A cold-start solve proved challenging: the naive MILO failed to find a feasible solution within 10.6 hours, and the optimality gap remained at 21.1\% after 72 hours. We therefore warm-start all subsequent runs using the feasible solution with all hot-idle and shutdown variables fixed to zero. This baseline solution has CAPEX of \$7.212B.

The principal finding is that hot-idling provides no measurable benefit for capacity planning under the base parameterization. In the hot-idle formulation, the warm-start solution was proven optimal, indicating that the model never selected hot-idle operation despite explicitly allowing it.

Table~\ref{tab:gap-cuts} and Figure~\ref{fig:hot-idle} summarize the effect of the bilinear cuts. For the hot-idle formulation, the cuts substantially tightened the root relaxation: the initial optimality gap fell from 21.2\% without cuts to 0.01\% with all three cuts. With all three cuts, optimality was proven at the root node. Without cuts, the model still converged, but required several hours of branch-and-bound. For the shutdown and full formulations, the impact of bilinear cuts was more mixed. Although the cuts improved the root relaxation, the resulting mixed-integer quadratic models became more difficult to solve. After 72 hours, the no-cut variants achieved tighter optimality gaps than the corresponding three-cut variants.

\begin{figure}[tbp]
    \centering
    \includegraphics[width=0.8\linewidth]{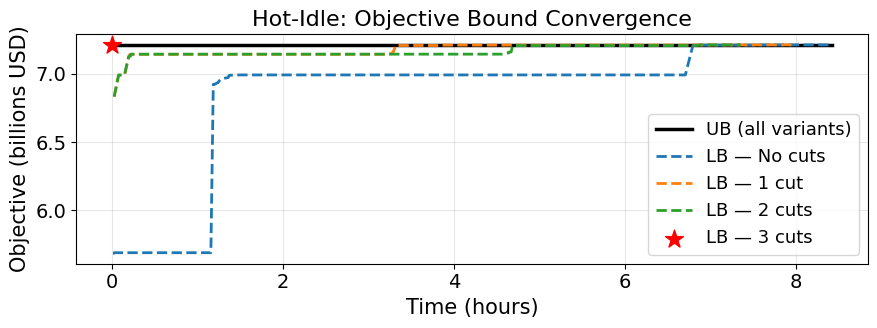}
    \caption{Solver progression for the hot-idle model, all with warm-starts. The warm-start baseline (blue) provides the initial upper bound, which was proved optimal. The no-cut variant converged in 6.8 hours, the 2-cut and 1-cut variants in 4.7 and 3.3 hours respectively, while the 3-cut variant proved optimality at the root node.}
    \label{fig:hot-idle}
\end{figure}

\begin{table}[tbp]
\centering
\caption{Optimality gaps at the root node and after extended solve times for varying numbers of cuts. Bold entries indicate the smallest gap within each model-time row.}
\label{tab:gap-cuts}
\begin{tabular}{llcccc}
\toprule
\textbf{Model} & \textbf{Time} 
& \textbf{No cuts} 
& \textbf{1 cut} 
& \textbf{2 cuts} 
& \textbf{3 cuts} \\
\midrule
Hot-idle    
    & Root node     & 21.20\% & 5.27\% & 5.27\% & \textbf{0.01\%} \\
    & 8 hours       & 0.01\%  & 0.01\% & 0.01\% & 0.01\% \\
\midrule
\multirow{3}{*}{Shutdown} 
    & Root node     & 21.10\% & 5.23\% & 5.23\% & \textbf{4.73\%} \\
    & 24 hours      & 4.07\% & 4.66\% & 4.94\% & \textbf{3.95\%} \\
    & 72 hours      & \textbf{2.56\%} & 4.45\% & 4.86\% & 3.95\% \\
\midrule
\multirow{3}{*}{Full} 
    & Root node     & 21.20\% & 5.27\% & 5.27\% & \textbf{5.25\%} \\
    & 24 hours      & \textbf{5.20\%} & 5.24\% & 5.24\% & 5.22\% \\
    & 72 hours      & \textbf{2.28\%} & 4.96\% & 5.14\% & 5.22\% \\
\bottomrule
\end{tabular}
\end{table}

Unlike hot-idling, shutdown flexibility produced a modest improvement. The best incumbent solution was obtained by the shutdown formulation with three cuts, achieving a CAPEX of \$7.159B, a 0.73\% reduction relative to the no-flexibility baseline. This solution retained a 3.95\% optimality gap after 72 hours, while the shutdown formulation without cuts produced a nearly identical incumbent and a smaller 2.56\% gap. The full formulation without cuts achieved a CAPEX of \$7.168B with a 2.28\% gap. We therefore interpret the results as evidence that shutdown flexibility provides some benefit, although its magnitude cannot be precisely quantified.

\subsubsection{Sensitivity Analysis of the No--Hot-Idling Solution}\label{sec:robustness}
We now examine the sensitivity of the no-hot-idling optimal solution to parameters that directly influence the value of operational flexibility. Specifically, we perturb the minimum HB load during normal operation, $\lambda_{HB}^{\min}$, the minimum hot-idle load, $\lambda_{HB}^{\mathrm{idle}}$, and the marginal CAPEX of hydrogen storage, $C_{H_2}$. Perturbations are applied only in directions favorable to hot-idling, namely increasing $\lambda_{HB}^{\min}$, decreasing $\lambda_{HB}^{\mathrm{idle}}$, and increasing $C_{H_2}$.

Table~\ref{tab:robustness_results} reports the resulting optimal CAPEX values. Changes to the hot-idle load and hydrogen-storage cost had no effect on the optimal solution. In contrast, increasing the minimum HB operating load created a modest benefit from operational flexibility, reducing CAPEX by 0.27\% at $\lambda_{HB}^{\min}=30\%$ and 3.25\% at $\lambda_{HB}^{\min}=60\%$.

Overall, the results suggest that hot-idling provides little value under the base parameterization and remains economically unattractive across a wide range of hydrogen-storage costs and hot-idle operating levels. Only when the minimum-load requirement becomes substantially more restrictive does operational flexibility begin to materially reduce system cost.

\begin{table}[tbp]
\centering
\caption{Sensitivity of optimal CAPEX to variations in the minimum HB load, hot-idle load, and hydrogen-storage cost.}
\label{tab:robustness_results}
\renewcommand{\arraystretch}{1.15}
\begin{tabular}{c l c c c c}
\toprule
\textbf{Parameter}
& \textbf{Change}
& \textbf{Value}
& $\boldsymbol{Z_0^\star}$
& $\boldsymbol{Z^\star}$
& $\boldsymbol{\frac{Z^\star-Z_0^\star}{Z_0^\star}}$ \\
\midrule
$\lambda_{\mathrm{HB}}^{\min}$
& Base       & $10\%$  & 7.212 & 7.212 & 0.00\% \\
& $+20\%$    & $30\%$  & 7.379 & 7.359 & $-0.27$\% \\
& $+40\%$    & $60\%$  & 7.669 & 7.420 & $-3.25$\% \\
\midrule
$\lambda_{\mathrm{HB}}^{\mathrm{idle}}$
& Base       & $75\%$  & 7.212 & 7.212 & 0.00\% \\
& $-20\%$    & $55\%$  & 7.212 & 7.212 & 0.00\% \\
& $-40\%$    & $35\%$  & 7.212 & 7.212 & 0.00\% \\
\midrule
$C_{H_2}$
& Base       & \$660/kg     & 7.212 & 7.212 & 0.00\% \\
& $\times 2$ & \$1{,}320/kg & 7.237 & 7.237 & 0.00\% \\
& $\times 5$ & \$3{,}300/kg & 7.312 & 7.312 & 0.00\% \\
\bottomrule
\end{tabular}
\end{table}\subsection{Strategic Model Performance under Perfect Forecasts}

This experiment evaluates whether the proposed scenario-reduced formulations produce effective capacity decisions. We compare the resulting designs against the full-horizon perfect-foresight benchmark and then assess their operational performance on unseen data using the rolling-horizon simulation model.

\subsubsection{Capacity Decisions}

To isolate the impact of scenario reduction, we assume perfect renewable forecasts and exclude the hot-idle and shutdown features of Section~\ref{section:mip}. For the scenario-reduced models, we select $n=5$ representative years, a parameter tuned to balance coverage and tractability.

\begin{table}[t]
\centering
\caption{Optimal capacity decisions across strategic model variants.}
\label{tab:strategic_comparison}
\begin{tabular}{lrrrr}
\toprule
\textbf{Component} & \textbf{Perfect-Foresight} & \multicolumn{3}{c}{\textbf{Scenario-Reduced}} \\
\cmidrule(lr){3-5}
 &  & \textbf{CLT} & \textbf{Budget} & \textbf{Aggregation} \\
\midrule
Solar [MW] & 1780 & 1676 & 1600 & 0 \\
Wind [MW] & 1926 & 2159 & 2224 & 2899 \\
Battery [MWh] & 2046 & 1395 & 716.4 & 407.1 \\
Hydrogen Storage [tons] & 80.69 & 31.64 & 24.36 & 33.37 \\
Desalination [$\mathrm{million\ m^3}$/year] & 6.320 & 6.398 & 6.321 & 7.447 \\
Electrolyzer [MW] & 1626 & 1667 & 1647 & 1771 \\
HB [MT NH$_3$/year] & 1.262 & 1.387 & 1.390 & 1.572 \\
\bottomrule
\end{tabular}
\end{table}
Table~\ref{tab:strategic_comparison} reports the optimal capacity decisions for each formulation. The LCOA objective values are omitted due to confidentiality restrictions. The CLT and Budget formulations produce qualitatively similar capacity plans with a diversified mix of wind, solar, and storage assets. In contrast, the constraint-aggregation formulation eliminates solar capacity entirely and relies almost exclusively on wind generation. By averaging renewable trajectories within clusters, the aggregation approach obscures prolonged low-generation events that are critical for satisfying the HB minimum-load requirements, leading to a markedly different capacity plan.
\subsubsection{Out-of-Sample Simulation}

To evaluate the capacity decisions, we simulate hourly operations over the 10-year test set using the rolling-horizon operational model. The prediction horizon and control horizon are set to 72 and 6 hours, respectively, and the storage reward coefficients are tuned to $(k_B,k_H)=(0.01,0.1)$.

The perfect-foresight, CLT, and Budget formulations remained feasible throughout the entire simulation horizon. Average ammonia production was 2876, 3052, and 3032 tons/day, respectively. In contrast, the constraint aggregation formulation violated the minimum-load requirements of the HB process and could not be operated feasibly, even under perfect forecasts. These results indicate that the proposed CLT and Budget formulations preserve sufficient temporal structure to produce feasible capacity plans, whereas simple constraint aggregation does not.

\subsection{Operational Model Performance under Forecast Error}

We compare the nominal, static-robust, and adaptive-robust operational models under forecast uncertainty. All simulations use the capacities obtained from the CLT strategic formulation (Table~\ref{tab:strategic_comparison}) and are performed over one year of unseen renewable data using a 12-hour prediction horizon and a 6-hour control horizon. Storage reward coefficients are tuned to $(k_B,k_H)=(0.001,0.01)$, reflecting the shorter 12-hour prediction horizon.

Forecasts are generated by perturbing the realized capacity factors with Gaussian noise:
\[
\overline{v}_{\texttt{source}}^{t}
=
v_{\texttt{source}}^{t}
-
u_{\texttt{source}}^{t},
\qquad
u_{\texttt{source}}^{t}
\sim
\mathcal N\!\left(0,(\sigma_{\texttt{source}}^t)^2\right).
\]
The forecast-error standard deviations are
\(
\sigma_s^t
=
\mathbf{1}{\{\overline v_s^t>0\}}\cdot\,0.02\sqrt t, \
\sigma_w^t
=
0.05\sqrt t.
\)
Thus, solar forecasts are deterministic at night, wind forecasts exhibit greater variability than solar forecasts, and uncertainty increases with forecast horizon.

Table~\ref{tab:simulation-operational} reports mean ammonia production and external grid usage to measure the trade-off between production and robustness. Additionally we report failure rates, which denote infeasibilities that cannot be recovered via the external grid intake. When a failure occurs, the corresponding control interval is skipped and the storage states are carried forward. Reported production and grid-usage statistics exclude failed intervals.

\begin{table}[tbp]
\centering
\caption{Operational simulation results for nominal, static-robust, and adaptive-robust models under varying robustness levels $\Omega$. Ammonia and grid usage are mean values over a one-year rolling horizon. Failures denote infeasibilities that cannot be recovered via external grid intake.}
\label{tab:simulation-operational}
\begin{tabular}{c c c c c c c}
\toprule
$\Omega$ & \multicolumn{2}{c}{\textbf{NH\textsubscript{3} Produced (tpd)}} & \multicolumn{2}{c}{\textbf{Grid Usage (MW)}} & \multicolumn{2}{c}{\textbf{Failure Rate (\%)}} \\
& Static & Adaptive & Static & Adaptive & Static & Adaptive \\
\midrule
0 (Nominal) & 2959 & 2959 & 59.25 & 59.25 & 0.0 & 0.0 \\
1           & 2744 & 2921 &  9.50 &  5.34 & 0.0 & 41.7 \\
2           & 2496 & 2854 &  2.49 & 18.81 & 0.0 &  8.7 \\
3           & 2237 & 2754 &  2.84 & 37.56 & 0.0 &  0.3 \\
\bottomrule
\end{tabular}
\end{table}

The nominal model ($\Omega=0$) achieves the highest ammonia production but relies heavily on external electricity, requiring an average of 59.3 MW of grid imports. Such dependence may compromise the economic and environmental benefits of green ammonia production. Increasing $\Omega$ reduces grid dependence at the expense of lower production. For example, the static-robust formulation with $\Omega=1$ reduces grid imports by 84\% while decreasing ammonia production by only 7.3\%. Across all robustness levels, the adaptive formulation consistently achieves higher production than the static formulation.

However, the adaptive formulation introduces a failure mode absent in the nominal and static formulations. In the nominal and static models, forecast errors affect only the energy-balance constraints \eqref{eq:operational-balance}, and renewable shortfalls can be recovered through additional grid imports. In contrast, affine decision rules make all flow variables depend on the uncertainty realization, causing nonnegativity, storage, and capacity constraints to depend on the forecast error as well. Consequently, realizations outside the uncertainty set may render the solution infeasible, and unlike \eqref{eq:operational-balance}, these violations cannot be repaired through additional grid intake. We refer to such events as \emph{brittle failures}. Recently, \citet{zhu_overfitting_2025} argued that this phenomenon is analogous to overfitting in machine learning, where a policy performs well within the assumed uncertainty model but generalizes poorly to out-of-set realizations. It is worth noting that the rolling-horizon framework already provides adaptivity through repeated forecast updates and re-optimization; thus, the gains reported here represent the incremental benefit of affine recourse beyond model predictive control. Nevertheless, brittleness can be substantially mitigated through uncertainty-set calibration. Increasing the robustness parameter reduces the failure rate from 41.7\% at $\Omega=1$ to 0.3\% at $\Omega=3$, while maintaining significantly higher production than the static formulation. Overall, the static formulation provides the most reliable operating policy, whereas the adaptive formulation offers higher production at the cost of greater probability of brittle failures.

\section{Conclusion}\label{section:conclusion}

We develop a robust optimization framework for green ammonia systems that integrates strategic capacity planning and operational dispatch under renewable uncertainty. At the strategic level, we formulate a capacity-planning model with flexible HB operation, including hot-idling and shutdown decisions, and develop an adversarial scenario-reduction framework for long-term renewable data. At the operational level, we formulate an adaptive robust rolling-horizon model under forecast uncertainty.

Computational results show that hot-idling provides limited value for strategic capacity planning under the considered parameters, while shutdown flexibility yields modest cost reductions. The proposed CLT and Budget scenario-reduction formulations produce feasible capacity plans under out-of-sample simulation, whereas constraint aggregation fails to do so. At the operational level, static-robust policies substantially reduce grid dependence, while adaptive policies achieve higher production but may exhibit brittle failures when forecast errors fall outside the uncertainty set.

\clearpage

\bibliographystyle{apalike}
\bibliography{references}
\clearpage

\appendix
\section*{Appendices}
\addcontentsline{toc}{section}{Appendices}

\section{Strategic Model}\label{app:strategic}
Table~\ref{tab:param-values} lists the parameter values used in the experiments.
\begin{table}[htbp]
\centering
\caption{Input parameters for the numerical experiments.}
\label{tab:param-values}
\small
\renewcommand{\arraystretch}{1.05}

\begin{tabular}{lll|lll}
\toprule
\multicolumn{3}{c|}{\textbf{Cost and storage parameters}}
&
\multicolumn{3}{c}{\textbf{Process and operating parameters}}
\\
\midrule
Symbol & Value & Units &
Symbol & Value & Units \\
\midrule

$C_s$ & 605,000 & \$/MW &
$\phi^e_{H_2^{\rm comp}}$ & 0.741 & ton/MWh \\

$C_w$ & 1,100,000 & \$/MW &
$\phi^e_{H_2}$ & 0.0192 & ton/MWh \\

$C_b$ & 275,000 & \$/MWh &
$\phi^e_{N_2}$ & 2.78 & ton/MWh \\

$C_e$ & 1,650,000 & \$/MW &
$\phi^e_{H_2O}$ & 285.7 & ton/MWh-m$^3$ \\

$C_{H_2}$ & 660,000 & \$/ton-H$_2$ &
$\phi^e_{NH_3}$ & 2.27 & ton/MWh \\

$C_{HB}$ & 816 & \$/ton-NH$_3$/yr &
$\phi^m_{H_2,H_2O}$ & 0.04 & -- \\

$C_d$ & 4 & \$/m$^3$/day &
$\phi^m_{NH_3,N_2}$ & 0.833 & -- \\

$C_{\rm trans}$ & 253 & M\$ &
$\phi^m_{NH_3,H_2}$ & 0.179 & -- \\

$C_{\rm port}$ & 200 & M\$ &

$\lambda_{HB}^{\min}$ & 10\% & of $K_{HB}$ \\

$\gamma$ & 6\% & /yr 
&$\lambda_{HB}^{e}$ & 60 & MWh / MT-NH$_3$ \\

$c_s$ & 1.5\% & of $C_s$ &
$\lambda_{HB}^{\rm idle}$ & 75\% & of $K_{HB}$ \\

$c_w$ & 2.5\% & of $C_w$ &
$\kappa$ & $10^6$ & ton NH$_3$ \\

$c_b$ & 4.2\% & of $C_b$ &
$\eta_{\rm RTE}$ & 85\% & -- \\

$c_e$ & 3.5\% & of $C_e$ &
$\eta_{\rm heat}$ & 99\% & -- \\

$c_{H_2}$ & 2\% & of $C_{H_2}$ &
$\Lambda_b$ & 20\% & of $K_b$/hr \\

$c_{HB}$ & 2\% & of $C_{HB}$ &
$\Lambda_{H_2}$ & 17.46 & ton/day \\

$S_{b,\rm init}$ & 50\% & of $K_b$ &
$\delta_s$ & 99.5\% & /yr \\

$S_{H_2,\rm init}$ & 50\% & of $K_{H_2}$ &
$\delta_w$ & 100\% & /yr \\

$S_{b,\max}$ & 90\% & of $K_b$ &
$\delta_b$ & 98\% & /yr \\

$S_{H_2,\max}$ & 90\% & of $K_{H_2}$ &
$\delta_e$ & 98.7\% & /yr \\

\bottomrule
\end{tabular}
\end{table}

\section{Operational Model}

\subsection{Probabilistic Guarantees}\label{app:prob-guarantee}
We now derive quantitative prescriptions for calibrating the robustness parameter~$\Omega$. For each constraint $l$,
\begin{equation}\label{eq:constraint-l}
\overline{\boldsymbol{a}}_l^\top \boldsymbol{K} 
+ \mathbf{1}_{\{l=\eqref{eq:operational-epigraph}\}} \psi 
+ \mathbf{1}_{\{l=(\eqref{eq:lp-model-energy-balance},t)\}} \boldsymbol{u}^\top \boldsymbol{K}_{\texttt{source}}^t 
+ \boldsymbol{d}_{x,l}^\top(\boldsymbol{z}_x + \boldsymbol{V}_x \boldsymbol{u}) 
+ \boldsymbol{d}_{w,l}^\top(\boldsymbol{z}_w + \boldsymbol{V}_w \boldsymbol{u}) 
\le b_l 
\end{equation}
with a fixed feasible solution to \eqref{eq:constraint-general}, we may require a guarantee
\begin{equation}\label{eq:prob-guarantee}
\mathbb{P}\left[\text{constraint $l$ satisfied}\right] \ge 1-\epsilon
\end{equation}
for some small $\epsilon > 0$.
In the robust optimization literature, probabilistic guarantees are well established for Gaussian perturbations and for distribution-free independent perturbations with bounded support in $[-\boldsymbol{1},\boldsymbol{1}]$. 
In our setting, the perturbations are bounded as 
$\boldsymbol{u}\in\boldsymbol{[L,U]}=[-\boldsymbol{\overline v},\,\boldsymbol{1}-\boldsymbol{\overline v}] 
\subseteq [-\boldsymbol{1},\boldsymbol{1}]$. 
We first restate the classical results as lemmas and then extend them as propositions to account for truncation and bounded supports in our model.

\paragraph{Gaussian case.}
We begin with the guarantee for Gaussian forecast errors, 
which relates the $(1-\epsilon)$ quantile of the multivariate normal 
to the corresponding ellipsoidal robust counterpart.
\begin{lemma}\label{lemma:gaussian}
    If $\boldsymbol{u}\sim\mathcal N(\boldsymbol{0},\Sigma)$, then 
\eqref{eq:prob-guarantee} holds if
\begin{equation}
\overline{\boldsymbol{a}}_l^\top \boldsymbol{K} 
+ \mathbf{1}_{\{l=\eqref{eq:operational-epigraph}\}} \psi 
+ \boldsymbol{d}_{x,l}^\top \boldsymbol{z}_x 
+ \boldsymbol{d}_{w,l}^\top \boldsymbol{z}_w  +
 \Omega_{1-\epsilon} \left\| \Sigma^{\frac12} \Big( 
   \mathbf{1}_{\{l=(\eqref{eq:lp-model-energy-balance},t)\}} \boldsymbol{K}_{\texttt{source}}^t  
   + V_x^\top \boldsymbol{d}_{x,l} 
   + V_w^\top \boldsymbol{d}_{w,l} 
\Big) \right\|_2 \le b_l , 
\end{equation}
where $\Omega_{1-\epsilon}$ is the $(1-\epsilon)$-quantile of the standard normal. 
Equivalently, this is the RC of \eqref{eq:constraint-general} with ellipsoidal uncertainty set
\[
\mathcal U_\text{ellipsoid}=\left\{\boldsymbol{u}:\ 
\bigl\|\Sigma^{-1/2}\boldsymbol{u}\bigr\|_2 \le \Omega_{1-\epsilon}\right\}.
\]
\end{lemma}
We now extend the guarantee to the truncated case relevant to bounded capacity factors.
\begin{proposition}[Truncated Gaussian]\label{prop:trunc-gauss}
Let $\boldsymbol{u}\sim\mathcal N(\boldsymbol{0},\Sigma)$ and define 
$p_{B}=\mathbb{P}(\boldsymbol{u}\in[-\boldsymbol{\overline v},\,\boldsymbol{1}-\boldsymbol{\overline v}])>0$.
Then the AARC \eqref{eq:aarc} satisfies the probabilistic guarantee \eqref{eq:prob-guarantee}
if the ellipsoid–box uncertainty set \eqref{eq:operational-uset} is constructed with 
$\Omega=\Omega_{\,1-\epsilon p_{B}}$.
\end{proposition}

\begin{proof}
Let $E$ denote the event that constraint $l$ is satisfied under the untruncated Gaussian $\boldsymbol{u}$, and let $B=\{\boldsymbol{u}\in[-\boldsymbol{\overline v},\,\boldsymbol{1}-\boldsymbol{\overline v}]\}$. The probability of constraint $l$ being satisfied under the truncated $\boldsymbol{u}$ is
\[
\mathbb{P}(E\mid B)
=\frac{\mathbb{P}(E\cap B)}{\mathbb{P}(B)}=\frac{\mathbb{P}(E) -\mathbb{P}(E\cap B^{\mathrm c})}{\mathbb{P}(B)}
\;\ge\;\frac{\mathbb{P}(E)-\mathbb{P}(B^{\mathrm c})}{\mathbb{P}(B)}.
\]
By Lemma~\ref{lemma:gaussian}, taking
$\Omega=\Omega_{1-\epsilon p_B}$ ensures $\mathbb{P}(E)\ge 1-\epsilon p_B$. Since
$\mathbb{P}(B^{\mathrm c})=1-p_B$, we get
\[
\mathbb{P}(E\mid B)
\;\ge\;\frac{1-\epsilon p_B-(1-p_B)}{p_B}
\;=\;1-\epsilon. \qedhere
\]
\end{proof}
\paragraph{Distribution-free case.}
When the Gaussian assumption may not hold, a looser but fully distribution-free sub-Gaussian bound can be established. 
This result assumes that perturbations are independent across time and source.
Such an approximation is reasonable if the forecasting model $\overline{\boldsymbol v}$ 
already captures temporal and cross-source correlations, leaving only independent residual noise.
\begin{lemma}\label{lemma:dist-free}
If $\boldsymbol{u}$ are independent, zero-mean, with support in $\boldsymbol{[-1,1]}$, then the AARC \eqref{eq:aarc} with $\boldsymbol{L} = \boldsymbol{-1},\; \boldsymbol{U} = \boldsymbol{1}$ ensures
\[\mathbb P[\text{constraint $l$ satisfied}]
\ge 1-\exp(-\Omega^2/2).\] Choosing $\Omega=\sqrt{2\ln(1/\epsilon)}$ yields 
\eqref{eq:prob-guarantee}.
\end{lemma}
Again, we extend the guarantee to the bounded support case relevant to capacity factors.
\begin{proposition}[Distribution-free sub-Gaussian bound under bounded independence]\label{prop:distfree}
Suppose $\boldsymbol{u}$ is independent, zero-mean, with support $\boldsymbol{[L,U] = [-\boldsymbol{\overline{v}},1-\boldsymbol{\overline{v}}]} \subseteq \boldsymbol{[-1,1]}$, where $\boldsymbol{\overline{v}}$ are the given point forecasts. Then the AARC \eqref{eq:aarc} satisfies guarantee \eqref{eq:prob-guarantee} if we set the ellipsoid-box uncertainty set \eqref{eq:operational-uset} with $\Omega=\sqrt{2\ln(1/\epsilon)}$.
\end{proposition}
\begin{proof}
Lemma~\ref{lemma:dist-free} provides us the guarantee for support in  $\boldsymbol{[-1,1]}$. Shrinking the box \eqref{eq:uset-capacity-bound} from $ \boldsymbol{[-1,1]}$ to $[-\boldsymbol{\overline{v}},\boldsymbol{1-\overline{v}}]$ does not decrease the  $1-\epsilon$ guarantee since  $\boldsymbol{u \in [-\boldsymbol{\overline{v}},1-\boldsymbol{\overline{v}}]}$ almost surely.
\end{proof}

\subsection{Proof of Proposition~\ref{prop:aarc} (AARC derivation)}\label{app:proof-aarc}
\begin{proof}
For a fixed constraint $l$, define
\[
c_l
=
\overline{\boldsymbol a}_l^\top \boldsymbol K
+
\delta_l^\psi \psi
+
\boldsymbol d_{x,l}^\top \boldsymbol z_x
+
\boldsymbol d_{w,l}^\top \boldsymbol z_w
, \quad
\boldsymbol q_l
=
\delta_l^e \boldsymbol K_{\mathrm{source}}^t
+
\boldsymbol V_x^\top\boldsymbol d_{x,l}
+
\boldsymbol V_w^\top\boldsymbol d_{w,l}.
\]
Then \eqref{eq:constraint-general} can be written as
\[
c_l+\boldsymbol q_l^\top\boldsymbol u\le b_l,
\qquad
\forall \boldsymbol u\in\mathcal U.
\]
Equivalently,
\[
c_l+
\max_{\boldsymbol u}
\left\{
\boldsymbol q_l^\top\boldsymbol u:
\boldsymbol L\le \boldsymbol u\le \boldsymbol U,\;
\|\Sigma^{-1/2}\boldsymbol u\|_2\le \Omega
\right\}
\le b_l.
\]

For any $\boldsymbol\alpha,\boldsymbol\beta\ge\boldsymbol 0$ satisfying the
box dualization, we have
\[
\boldsymbol q_l^\top\boldsymbol u
\le
\boldsymbol U^\top\boldsymbol\beta
-
\boldsymbol L^\top\boldsymbol\alpha
+
(\boldsymbol q_l-\boldsymbol\beta+\boldsymbol\alpha)^\top\boldsymbol u.
\]
Maximizing the remaining linear term over the ellipsoid gives
\[
\max_{\|\Sigma^{-1/2}\boldsymbol u\|_2\le \Omega}
(\boldsymbol q_l-\boldsymbol\beta+\boldsymbol\alpha)^\top\boldsymbol u
=
\Omega
\left\|
\Sigma^{1/2}
(\boldsymbol q_l-\boldsymbol\beta+\boldsymbol\alpha)
\right\|_2.
\]
Therefore, the robust constraint is equivalent to the existence of
$\boldsymbol\alpha,\boldsymbol\beta\ge\boldsymbol 0$ such that
\[
c_l
+
\boldsymbol U^\top\boldsymbol\beta
-
\boldsymbol L^\top\boldsymbol\alpha
+
\Omega
\left\|
\Sigma^{1/2}
(\boldsymbol q_l-\boldsymbol\beta+\boldsymbol\alpha)
\right\|_2
\le b_l.
\]
Substituting the definitions of $c_l$ and $\boldsymbol q_l$ yields
\eqref{eq:aarc}.
\end{proof}

\end{document}